\documentclass[preprint]{elsarticle}

\usepackage{amssymb}
\usepackage{amsthm}

\usepackage{amsmath,amsfonts}
\usepackage{graphicx}%

\theoremstyle{plain}
\newtheorem{Thm}{Theorem}
\newtheorem{Pro}{Proposition}
\newtheorem{Lem}{Lemma}
\newtheorem{Cor}{Corollary}
\theoremstyle{definition}
\newtheorem{Def}{Definition}
\theoremstyle{remark}
\newtheorem{Rem}{Remark}
\newtheorem{Exa}{Example}
\usepackage{hyperref}
\usepackage{color}
\usepackage[small]{caption}
\usepackage{pstricks, pst-node, pst-plot, pst-circ}
\usepackage{moredefs}

\newcommand{\LtR}{ L^2 (\mathbb{R})}
 
 \def\RR{\mathbb{R}}
 
 \def\CC{\mathbb{C}}
 \def\ZZ{\mathbb{Z}}
 \def\supp{\operatorname{supp}}
 \def\sgn{\operatorname{sgn}}
 \def\bd{\mathbf}
 
 \def\NN{\mathbb{N}}

\begin{document}

\begin{frontmatter}

\title{Structure of nonstationary Gabor frames and their dual systems}
\author{Nicki Holighaus\fnref{fn1}}
\ead{nicki.holighaus@kfs.oeaw.ac.at}
\fntext[fn1]{Tel. +43 1 51581-2516, Fax. +43 1 51581 2530}
\address{Acoustics Research Institute\\
		Austrian Academy of Sciences\\
		Wohllebengasse 12--14, A-1040 Vienna, Austria}
\date{\today}

  \begin{abstract}
    We investigate the structural properties of dual systems for nonstationary Gabor frames. In particular, we prove that some inverse nonstationary Gabor frame operators admit a Walnut-like representation, i.e. the operator acting on a function can be described by weighted translates of that function, even when the original frame operator is not diagonal. In this case, which only occurs when compactly supported window functions are used, 
    the canonical dual frame partially inherits the structure of the original frame, with differences that we describe in detail. Moreover, we determine a sufficient condition for a pair of nonstationary Gabor frames to form dual frames. The equivalence of this condition to the duality of the involved systems is shown under some weak restrictions. It is then applied in a simple setup, to prove the existence of dual pairs of nonstationary Gabor systems with non-diagonal frame operator.    
    A discussion of the results, restricted to the classical case of regular Gabor systems, precedes the statement of the general results. Here, we also explore a connection to recent work of Christensen, Kim and Kim on Gabor frames with compactly supported window function.
  \end{abstract}


\end{frontmatter}
  
\section{Introduction}\label{sec:intro}
    In this contribution, we investigate the properties of adaptive time-frequency systems that generalize classical Gabor systems. Although some of the presented results apply in a more general setting, the focus is  on time-frequency systems with compactly supported generators.
    
    Let $g\in\LtR$ and $a,b\in\RR^+$. The corresponding \emph{Gabor system}~\cite{gr01,ga46} $\mathcal{G}(g,a,b)$ is the set of functions
    \begin{equation}\label{eq:gabsys}
      g_{m,n}(t) = \bd{M}_{mb}\bd{T}_{na} g(t) = g(t-na)e^{2\pi i mbt},~\forall~m,n\in\ZZ.
    \end{equation}
    The prototype function $g$ is also called \emph{window} or \emph{generator function}.
    Of particular interest are systems that allow for stable, perfect reconstruction of any function $f\in\LtR$ from the system coefficients, 
    given by inner products with the system elements. Such systems are generally called \emph{frames}~\cite{dusc52,ch03} or, when they are of the form $\mathcal{G}(g,a,b)$, \emph{Gabor frames}.
    For any frame, there exists a possibly non-unique \emph{dual frame} that enables the aforementioned perfect reconstruction. Gabor frames $\mathcal{G}(g,a,b)$ possess the nice property that, due to their highly structured nature, the existence of a dual frame with the same structure, $\mathcal{G}(h,a,b)$ for some $h\in\LtR$, is guaranteed. This inheritance of structure from the original frame by a dual frame does not hold for more general frames and is one of the reasons why Gabor frames are so convenient to work with. One such dual frame is the \emph{canonical dual}, obtained by applying the inverse \emph{frame operator}, cf. Section \ref{sec:prelims} for details, to the frame elements.
    
    One of the early and most prevalent results in the field is the theory of painless non-orthogonal expansions~\cite{dagrme86}, where the authors determine a simple necessary and sufficient condition for Gabor frames $\mathcal{G}(g,a,b)$ with compactly supported generator and dense frequency sampling, i.e. small frequency step $b$, to constitute a frame. Then, the frame operator is diagonal and thus easily inverted and the canonical dual generator $\tilde{g}$ has the same support as $g$. This setting is often referred to as the \emph{painless case}.
    
    In applications, frames generated from compactly supported window functions are of particular interest, because they allow for the most efficient computation of the frame coefficients and reconstruction. Compact support of the frame generators is also crucial for real-time implementation. Thus, the investigation of such frames beyond the painless case is an active field, see e.g. \cite{boja00,ch06,chsu08,chki10,la09-2} and \cite{chkiki10,chkiki12}. In the latter two articles, Christensen, Kim and Kim prove that for any Gabor frame with $\supp(g)\subseteq [1,1]$, $a = 1$ and $b\in ]1/2,1[$, there exists a dual Gabor frame generated by a window supported on some compact set dependent only on the magnitude of $b$, cf. \cite[Theorem 2.1, Lemma 3.2]{chkiki10}. In fact, they show in \cite{chkiki12} that the support condition in \cite{chkiki10} can be further improved, for a large class of window functions $g$. This is also reflected in our own results in Section \ref{sec:regcas} albeit 
    recovering only a special case of the results in \cite{chkiki12}. The results in this manuscript are somewhat complementary to those of Christensen, Kim and Kim. To allow for a comparison, we recall some results from \cite{chkiki10} in Section \ref{sec:regcas}. 

    More results on the support of dual Gabor frames are due to Gr\"ochenig and St\"ockler~\cite[Theorem 9]{grst11}. They prove the existence of dual frames with compactly
    supported, piecewise continuous generator for $\mathcal{G}(g,a,b)$ with $g$ a totally positive function of finite type. While the class of functions treated by Gr\"ochenig and St\"ockler 
    is quite different from the compactly supported functions in this contribution, the support size of the dual generator grows proportionally to the quotient $\frac{ab}{1-ab}$ in 
    both cases.\\

    Although Gabor systems possess a number of useful mathematical properties, that lead to a deep, yet accessible theory, the fixed time-frequency resolution and sampling strategy they provide is often debated as too restrictive for practical purposes. As a result, various generalizations have been proposed, to provide more flexible sampling strategies or varying window functions. Methods that allow for prefect reconstruction with flexible sampling and varying windows are scarce, however. One construction that unites these desirable properties are \emph{nonstationary Gabor} (NSG) systems, first proposed by Jaillet~\cite{ja05-2}. 
    While classical Gabor systems are constructed from regular translations and modulations, NSG systems are generated by a countable set of window functions and modulations thereof.
    Explicitly, we associate a sequence of pairs $\mathcal{G}(\bd{g},\bd{b}):=(g_n,b_n)_{n\in\ZZ}$, $g_n\in\LtR$ and $b_n\in\RR^+$, with the set of functions
    \begin{equation}\label{eq:nsgab}
      g_{m,n}(t) = \bd{M}_{mb_n}g_n(t) = g_n(t)e^{2\pi imb_n t},\quad \text{for all } m,n\in\ZZ.
    \end{equation}
    If $\mathcal{G}(\bd{g},\bd{b})$ constitutes a frame, we call it a \emph{nonstationary Gabor frame}. 
    Note that a nonstationary Gabor system with $b_n = b$ and $g_n = T_{na}g$ for all $n\in\ZZ$ with $g\in\LtR$ and $a,b\in\RR^+$ is a Gabor system.

    Nonstationary Gabor frames combine the adaptivity of \emph{local Fourier bases}~\cite{ma92-2,au94} with the flexibility of redundant systems to provide a powerful framework for time-frequency representations.
    Much like Gabor frames give rise to \emph{Wilson bases}~\cite{dajajo91,fegrwa92,behewa98,bi03}, local Fourier bases can be constructed from NSG frames, although the more intricate properties of their relationship have yet to be investigated. 

    State of the art results on \emph{nonstationary Gabor frames} are collected in \cite{badohojave11}, where an extension of the painless case
    to nonstationary Gabor systems is also given. 
    
    The young theory of nonstationary Gabor frames beyond the painless case is currently being developed~\cite{doma11,doma12-2}, while the 
    painless construction is being used in realizing various time- or frequency-adaptive transforms~\cite{prruwo10,lirororo11,liromaroro13,badohojave11,nebahoso13,dogrhove12}.\\
    
    Note that, in contrast to regular Gabor frames, the existence of a dual frame with the same structure, i.e. comprised of window functions $h_n$ and modulation 
    parameters $b_n$, is not guaranteed for general NSG frames. Indeed, one of the central results in this manuscript details the structure of the canonical dual system under certain
    restrictions. These restrictions, concerning the support and overlap of the window functions $g_n$ and the modulation parameters $b_n$, guarantee compact support for the elements
    of the canonical dual frame and a certain modulation and phase shift structure, detailed in Section \ref{sec:nsgrap}. This structure can be deduced from that of the inverse frame operator,
    which is in turn determined using the Walnut representation of the NSG frame operator and the Neumann series representation of its inverse. 
    
    Further, we obtain a duality condition, sufficient for pairs of dual nonstationary Gabor systems $\mathcal{G}(\bd{g},\bd{b})$ and $\mathcal{G}(\bd{h},\bd{b})$ to constitute dual frames. Under weak
    assumptions on $g_n$ and $b_n$, we are able to show equivalence of this condition with duality of $\mathcal{G}(\bd{g},\bd{b})$ and $\mathcal{G}(\bd{h},\bd{b})$. For a fixed NSG frame $\mathcal{G}(\bd{g},\bd{b})$, these equations might not be solvable, i.e. a dual system of the form $\mathcal{G}(\bd{h},\bd{b})$ may not even exist. We determine a simple, yet somewhat restrictive, condition on $\mathcal{G}(\bd{g},\bd{b})$, such that the duality conditions are solvable. 
    
    Our results apply to the classical Gabor case by choosing $b_n = b$ and $g_n = T_{na}g$ to describe the support of the canonical dual window $\tilde{g}$ in the setting considered in \cite{chkiki10} and \cite{chkiki12}, complementing the results therein. By restricting the duality conditions for NSG systems in that way, we recover the famous duality conditions for Gabor systems~\cite{rosh95,rosh97,ja98} and
    a simple special case of a result in \cite{chkiki12}.
    
    The rest of the paper is organized as follows. Section \ref{sec:prelims} introduces the basic concepts and notation used, while Section \ref{sec:walnutlike} introduces the Walnut representation of the nonstationary Gabor frame operator and related concepts. In Section \ref{sec:regcas}, we state our results for the regular Gabor case and compare them to the results of Christensen, Kim and Kim~\cite{chkiki10,chkiki12}. Section \ref{sec:nsgrap} presents our results in their general form for nonstationary Gabor frames. The proof of Theorem \ref{thm:mainres} is postponed to Section \ref{sec:prodis} due to its lengthy nature. Section \ref{sec:conclu} concludes the paper with a summary of the results and an outlook.

%%% ---------------------------------------------------------------------------------------------------------------
%%% -----------------------------------------Preliminaries---------------------------------------------------------    
%%% ---------------------------------------------------------------------------------------------------------------    
    \section{Preliminaries}\label{sec:prelims}
    
    Before we state our results, some basic notions have to be clarified. In particular, we work with certain structured time-frequency dictionaries and the corresponding frame-related operators. By $\|\cdot \|$ we denote the $L^2$-norm and by $\|\cdot\|_{op}$ the operator norm. Furthermore we use the restriction $f\mid_I$ of $f\in\LtR$ to the interval $I$, the characteristic function $\chi_M$ of the set $M\subseteq\RR$ and the open $L^2(\RR)$-ball $B_{\delta}(t)$ around $t$ with radius $\delta$. The essential support, i.e. the support up to sets of measure zero, of $f\in\LtR$ is denoted by $\supp(f)$, the Lebesgue measure of a set $M\subseteq\RR$ by $\mu(M)$. Finally, we use the sign function $\sgn(t)$.
    
    Let $\Phi := (\phi_\lambda\in\LtR )_{\lambda\in\Lambda}$, with a countable index set $\Lambda$, be a system of generator functions. If there exist $0 < A\leq B < \infty$, such that
    \begin{equation}\label{eq:frameeq}
      A\|f\|^2 \leq \sum\limits_{\lambda\in\Lambda} |\langle f,\phi_\lambda\rangle|^2 \leq B\|f\|^2,\quad \forall~ f\in\LtR,
    \end{equation}
    then we call $\Phi$ a \emph{frame for $\LtR$}. If $\Phi$ satisfies the upper bound, but not necessarily the lower bound, $\Phi$ is a \emph{Bessel sequence for $\LtR$}. Since all 
    frames and Bessel sequences in this contribution are considered over $\LtR$, we will usually omit the reference to the function space from now on.

    The basic operators associated to frames are the analysis and synthesis operators $\bd{C}_{\Phi}:\LtR \rightarrow \ell^2(\Lambda)$ and $\bd{D}_{\Phi}:\ell^2(\Lambda) \rightarrow \LtR$ 
    defined by
    \begin{equation}\label{eq:ana_syn}
      \bd{C}_{\Phi}f(\lambda) = \langle f,\phi_\lambda\rangle \quad \text{and} \quad \bd{D}_{\Phi}c = \sum_{\lambda\in\Lambda} c(\lambda)\phi_\lambda,
    \end{equation}
    for all $f\in\LtR, c\in \ell^2(\Lambda)$. For two systems $\bd{\Phi}$ and $\bd{\Psi}$ with the same index set $\Lambda$, the mutual \emph{frame-type operator} $\bd{S}_{\Phi,\Psi} := \bd{D}_{\Psi}\bd{C}_{\Phi}$, 
    is given by
    \begin{equation}\label{eq:frametypeop}
      \bd{S}_{\Phi,\Psi}f(\lambda) = \sum_{\lambda\in\Lambda}\langle f,\phi_\lambda\rangle\psi_\lambda,\ \forall~f\in\LtR.
    \end{equation}
    If $\Phi=\Psi$, then we call $\bd{S}_{\Phi,\Phi} = \bd{S}_{\Phi}$ the \emph{frame operator} associated to $\Phi$.
    If $\Phi$ is a Bessel sequence with Bessel bound $B$, then $\bd{C}_{\Phi}$ and $\bd{D}_{\Phi}$ are bounded operators, i.e. $\|\bd{C}_{\Phi}\|^2_{op} = \| \bd{D}_{\Phi}\|^2_{op} \leq B$. If $\Phi$ also satisfies the lower frame inequality, then $\bd{S}_{\Phi}$ is bounded above and below: $A \leq \|\bd{S}_{\Phi}\|_{op} \leq B$. Analysis, synthesis and frame operators are well-defined for systems $\Phi$ violating the upper and/or lower frame condition, but no longer bounded above or below, respectively. Whenever attribution of the operator to a frame is clear, we will omit the subscript.

    For any frame $\Phi := \{\phi_\lambda\in\LtR \}_{\lambda\in\Lambda}$ with frame bounds $A,B$, there exists a, possibly non-unique, dual frame $\Psi := \{\psi_\lambda\in\LtR \}_{\lambda\in\Lambda}$ such that 
    \begin{equation}
      f = \bd{D}_{\Psi}\bd{C}_{\Phi}f = \bd{D}_{\Phi}\bd{C}_{\Psi}f,\quad \forall~f\in\LtR.
    \end{equation}

    In particular the \emph{canonical dual frame} is given by applying the inverse frame operator $\bd{S_{\Phi}}^{-1}$ to the frame elements:
    \begin{equation}
      \widetilde{\Phi} := (\tilde{\phi}_\lambda = \bd{S}_{\Phi}^{-1}\phi_\lambda)_{\lambda\in\Lambda}.
    \end{equation}
    
    For any frame, the inverse frame operator admits a Neumann series representation~\cite{dusc52}. Let $0<A\leq B<\infty$ be the optimal frame bounds, then
    \begin{equation}\label{eq:neumann}
      \bd{S}^{-1} = 2/(A+B) \sum_{j=0}^\infty (\bd{I}-2\bd{S}/(A+B))^j,
    \end{equation}
    where $\bd{I}$ denotes the identity operator. The normalization factor $2/(A+B)$ yields $\|\bd{I}-2\bd{S}/(A+B)\|_{op} \leq \frac{B-A}{B+A} < 1$ and the fastest convergence among all possible choices~\cite{gr93,li98-1,jaso07}.

    In this contribution, we are interested in Gabor and nonstationary Gabor systems, i.e. systems of the form $\mathcal{G}(g,a,b)$~\eqref{eq:gabsys} or $\mathcal{G}(\bd{g},\bd{b})$~\eqref{eq:nsgab},   
    that constitute frames or Bessel sequences. We associate $\bd{g}$ and $\bd{b}$ with the sequences $(g_n)_n$ and $(b_n)_n$, respectively.
    
    For a (nonstationary) Gabor system generated from compactly supported window functions with dense frequency sampling, a conveniently simple condition exists that is equivalent to the frame property~\cite{dagrme86,badohojave11}. More
    explicitly, let $\mathcal{G}(\bd{g},\bd{b})$ be a nonstationary Gabor system as per \eqref{eq:nsgab}, such that for all $n\in\ZZ$, some constants $c_n,d_n\in\RR$ exist with $\supp(g_n) \subseteq [c_n,d_n]$ and
    $b_n^{-1} \geq d_n-c_n$. Then $\mathcal{G}(\bd{g},\bd{b})$ forms a frame, with frame bounds $A$ and $B$, if and only if
    \begin{equation}\label{eq:nspainless}
      0 < A \leq \sum_n \frac{1}{b_n} |g_n|^2 \leq B < \infty, \text{ almost everywhere. }
    \end{equation}
    This setup is usually referred to as the \emph{painless case} and $\mathcal{G}(\bd{g},\bd{b})$ is called a \emph{painless system}, since the frame operator is diagonal and easily invertible. Moreover, the canonical dual frame is of the form $\mathcal{G}(\tilde{\bd{g}},\bd{b})$, with $\supp(\widetilde{g_n})\subseteq [c_n,d_n]$.  Independent of $\supp(g_n)$, the former is also true whenever $b_n = b$ for all $n\in\ZZ$, in particular for regular Gabor frames.
    
%%% ---------------------------------------------------------------------------------------------------------------
%%% -----------------------------------------Walnut reps-----------------------------------------------------------    
%%% ---------------------------------------------------------------------------------------------------------------
    
  \section{Walnut and Walnut-like representations}\label{sec:walnutlike}
  
    Both regular and and nonstationary Gabor frame-type operators admit a so-called \emph{Walnut representation}, i.e. a representation purely in terms of translates of the frame generators and the function to which the operator is applied. The Walnut representation for nonstationary Gabor frames has only recently been rigorously proven %(->12) 
    for systems constructed from window functions in the Wiener space~\cite{doma11}. Here, we also use a variant for Bessel sequences. For the proof, we refer the interested reader 
    to \cite{doma11}, since the Bessel case only requires minor modifications.
    
    \begin{Def}
      The Wiener space $W(L^\infty,\ell^1)$ is the space of functions $f\in L^\infty(\RR)$ such that 
      \begin{equation}
	\|f\|_{W(L^\infty,\ell^1)} := \sum_{k\in\ZZ} \mathop{\operatorname{ess~sup}}\limits_{t\in [0,1]} |f(t+k)| < \infty.
      \end{equation}
    \end{Def}
    
    \begin{Pro}[\cite{doma11}]\label{pro:walnut}%(->12)]
      Let $\mathcal{G}(\bd{g},\bd{b})$ and $\mathcal{G}(\bd{h},\bd{b})$ be nonstationary Gabor systems with $b_n\in\RR^+$ and $g_n\in\LtR$, for all $n\in\ZZ$. If either
      \begin{enumerate}
      \item[(i)] $g_n,~h_n\in W(L^\infty,\ell^1)$ for all $n\in\ZZ$, or
      \item[(ii)] $\mathcal{G}(\bd{g},\bd{b})$ and $\mathcal{G}(\bd{h},\bd{b})$ are Bessel sequences, 
      \end{enumerate}
      then the associated frame-type operator $\bd{S}_{\bd{g},\bd{h},\bd{b}} := \bd{D}_{\mathcal{G}(\bd{h},\bd{b})}\bd{C}_{\mathcal{G}(\bd{g},\bd{b})}$ admits a Walnut representation of the form
      \begin{equation}\label{eq:walnut}
	\bd{S}_{\bd{g},\bd{h},\bd{b}}f = \sum_{n,k\in\ZZ} b_n^{-1}h_n\bd{T}_{kb_n^{-1}}\overline{g_n}\bd{T}_{kb_n^{-1}}f, \quad \text{for all } f\in\LtR.
      \end{equation}
      Substituting $b_n$ by $b$ and $g_n$ by $\bd{T}_{na} g$ for all $n\in\ZZ$ yields the Walnut representation of Gabor frame-type operators. Setting $\bd{h}=\bd{g}$ yields the Walnut representation for the frame operator.
    \end{Pro}

    The Walnut representation shows that the frame operator maps a function $f\in\LtR$ onto a sum of weighted, translated copies of itself, where the weight functions are given by $\omega_{n,k} := b_n^{-1}g_n\bd{T}_{-kb_n^{-1}}\overline{g_n}$, for all $n,k\in\ZZ$ and the corresponding translates are $\bd{T}_{-kb_n^{-1}}$.

    The \emph{painless case} result \eqref{eq:nspainless} can be derived from the Walnut representation easily: If $\supp(g_n) \subseteq [c_n,d_n]$ and $b_n \leq \frac{1}{d_n-c_n}$,
    then $\omega_{n,k} = b_n^{-1}g_n\bd{T}_{-kb_n^{-1}}\overline{g_n} \equiv 0$ for all $k\neq 0$ and thus $\bd{S}$ is diagonal. Furthermore, boundedness of the sum in \eqref{eq:nspainless} is a necessary condition
    for any NSG system to constitute a Bessel sequence, 
    
    \begin{Pro}\label{pro:diagsbound}
      Let $\mathcal{G}(\bd{g},\bd{b})$ and $\mathcal{G}(\bd{h},\bd{b})$ be nonstationary Gabor Bessel sequences with $b_n\in\RR^+$ and $g_n\in\LtR$, for all $n\in\ZZ$. Let $B$ be a joint Bessel bound of 
      $\mathcal{G}(\bd{g},\bd{b})$ and $\mathcal{G}(\bd{h},\bd{b})$. Then 
      \begin{equation}\label{eq:boundeddiags}
	\sum_{n\in\ZZ} b^{-1}_n |h_n\bd{T}_x \overline{g_n}| \leq B\quad \text{a.e.}
      \end{equation}
      In particular $\sum_{n\in\ZZ} b^{-1}_n |g_n|^2 \leq B$ almost everywhere. 
    \end{Pro}
    \begin{proof}
      To prove $\sum_{n\in\ZZ} b^{-1}_n |g_n|^2 \leq B$, we retrace the steps of a proof by Chui and Shi~\cite{chsh93} for Wavelet frames. By the Bessel property of $\mathcal{G}(\bd{g},\bd{b})$ and Plancherel's equality for Fourier series,
      \begin{align*}
	B\|f\|^2 \geq \sum_{n,k\in\ZZ} |\langle f,g_{n,k} \rangle|^2 & = \sum_{n,k\in\ZZ} \left| \int_0^{b_n^{-1}} \sum_{l\in\ZZ} \bd{T}_{lb_n^{-1}}f(t)\bd{T}_{lb_n^{-1}}\overline{g_n(t)}e^{-2\pi ikb_n t}~dt\right|^2 \\
	& = \sum_{n\in\ZZ} b_n^{-1} \int_0^{b_n^{-1}} \left|\sum_{l\in\ZZ} \bd{T}_{lb_n^{-1}}f(t)\bd{T}_{lb_n^{-1}}\overline{g_n(t)}\right|^2~dt. \\
      \end{align*}
      Observe $b_n^{-1}$-periodicity of the integrand. For all $0<N\in\ZZ$, we can choose some $\epsilon > 0$, such that for all $t_0\in\RR$ and $f = \sqrt{2\epsilon}^{-1}\chi_{[t_0-\epsilon,t_0+\epsilon]}$
      \begin{align*}
	\lefteqn{\sum_{n=-N}^N \frac{1}{b_n} \int_{t_0-b_n^{-1}/2}^{t_0+b_n^{-1}/2} \left|\sum_{l\in\ZZ} \bd{T}_{lb_n^{-1}}f(t)\bd{T}_{lb_n^{-1}}\overline{g_n(t)}\right|^2~dt} \\
	& = \sum_{n=-N}^N \frac{1}{2\epsilon b_n} \int_{t_0-\epsilon}^{t_0+\epsilon} |g_n(t)|^2~dt \leq B\|f\| = B
      \end{align*}
      holds. Subsequently taking limits over $\epsilon$ and $N$ proves $\sum_{n\in\ZZ} b^{-1}_n |g_n|^2 \leq B$ almost everywhere. 
      
      The general case follows by Cauchy-Schwarz' inequality:
    \begin{equation}
      \sum_{n\in\ZZ} b_n^{-1}|h_n\bd{T}_x\overline{g_n}| \leq \Big(\sum_n b_n^{-1}|h_n|^2 \sum_l b_l^{-1}|g_l|^2\Big)^{1/2} \leq B, 
    \end{equation}
    for all $x\in\RR$.
    \end{proof}

    The Walnut representation is a very handy tool, describing the action of NSG frame operators in an intuitive way. We would like to use a slightly more general definition, though.
    
    \begin{Def}\label{def:wallike}
      Let $\Lambda$ be a countable index set, $X$ a dense subspace of $\LtR$ and $\bd{W} : \LtR \rightarrow \LtR$ a bounded linear operator. If sequences $(\omega_\lambda \in L^\infty(\RR))_{\lambda\in\Lambda}$ and $(a_\lambda\in\RR)_{\lambda\in\Lambda}$ exist such that 
      \begin{equation}\label{eq:wallike}
	\bd{W}f = \sum_{\lambda\in\Lambda}\omega_\lambda\bd{T}_{a_\lambda}f, \text{ for all } f\in X
      \end{equation}
      and the sum on the right-hand side is unconditionally convergent,
      then we say that $\bd{W}$ has a \emph{Walnut-like representation} with \emph{weights} $\omega_\lambda$ and \emph{translation constants} $a_\lambda$.
    \end{Def}
    
    For regular Gabor systems, the Walnut representation has been shown to be absolutely convergent by Janssen~\cite{ja98}. For more general NSG systems, we discuss an alternate Walnut-like representation of the nonstationary Gabor frame operator and its unconditional convergence in Section \ref{ssec:duality}.

    Under weak additional assumptions, we can show that in fact, the weights corresponding to a fixed translate of $f$ in \eqref{eq:wallike} are bounded by the operator norm of $\bd{W}$.
    
    \begin{Lem}\label{lem:sepeqbnd}
      Let $\bd{W}:\LtR\mapsto\LtR$ be a bounded linear operator with Walnut-like representation. % as per Definition \ref{def:wallike}. 
      If $\|\bd{W}\|_{op} = C < \infty$ and for all $c,d\in\RR$ with $c< d$, $\{a_\lambda~:~ \omega_\lambda\mid_{[c,d]}\neq 0\}_{\lambda\in\Lambda}$ is free of accumulation points, then 
      \begin{equation}\label{eq:bounded}
	\Big| \sum_{\substack{\lambda\in\Lambda \\ a_\lambda = a_{\lambda_0}}} \omega_\lambda \Big| \leq C\quad \text{a.e.},
      \end{equation}
      for all $\lambda_0\in\Lambda$.
    \end{Lem}
    \begin{proof}
      Without loss of generality, assume 
      \begin{equation*}
	\sum_{\substack{\lambda\in\Lambda \\ a_\lambda = a_{\lambda_0}}} \omega_\lambda \geq C_0 > C \text{ a.e. on } M \text{ with } \mu(M)>0.
      \end{equation*}
      Then for all $\delta > 0$, there exists $l\in\ZZ$, such that $\mu(M \cap B_{\delta}(2l\delta)) > 0$. Furthermore, since $\{a_\lambda~:~ \omega_\lambda\mid_{[c,d]}\neq 0\}_{\lambda\in\Lambda}$ has no accumulation points for all $c<d$, there exists for any fixed $l\in\ZZ$ a $\delta > 0$ such that 
      \begin{equation*}
	\{a_\lambda~:~ \omega_\lambda\mid_{B_{\delta}(2l\delta)}\neq 0\}_{\lambda\in\Lambda}\cap B_{2\delta}(a_{\lambda_0}) = \{a_{\lambda_0}\}.
      \end{equation*}
      Choose $l\in\ZZ$ and $\delta > 0$ such that the above equation holds and $\mu(M_l) > 0$, with $M_l = M \cap B_{\delta_0}(2l\delta_0)$. Take $f = \chi_{M_l-a_{\lambda_0}}$. If $f\in X$, then
      \begin{equation*}
	\left|\bd{W}f\mid_{B_{\delta_0}(2l\delta_0)}\right| \geq C_0|\bd{T}_{a_{\lambda_0}}f| \quad \Rightarrow \quad \|\bd{W}f\| \geq C_0\|f\|,
      \end{equation*}
      contradicting $\|\bd{W}\|_{op} = C < C_0$. If $f\notin X$, construct a sequence $(f_n \in X)_{n\in\NN}$ converging to $f$. For such a sequence, some $n_0\in\NN$ exists, such that $\|\bd{W}f_n\| > C\|f_n\|$, for all $n\geq n_0$.
    \end{proof}
    
    \begin{Rem}\label{rem:sumboundedness}
      If on the other hand, $\bd{W}:\LtR\mapsto\LtR$ is linear and $\bd{W}f$ can be written in the form \eqref{eq:wallike} for all $f$ in a dense subspace of $\LtR$, then $\bd{W}$ is guaranteed to be a bounded linear operator if $\sum_{\lambda\in\Lambda} \|\omega_\lambda\|_{\infty} < \infty$, by Cauchy-Schwarz' inequality.
    \end{Rem}

%%% ---------------------------------------------------------------------------------------------------------------
%%% -----------------------------------------Regular case----------------------------------------------------------    
%%% ---------------------------------------------------------------------------------------------------------------    
    
  \section{Results in the regular case}\label{sec:regcas}

	In this section, we recall a result of Christensen, Kim and Kim~\cite{chkiki10} and state our own results in a simplified form for regular Gabor frames. Thus, this section demonstrates the application of our results to a classical setting and eases the reader into the technicalities necessary for the description of the general case. 
    The results discussed herein are special cases of and follow directly from the results presented in Section \ref{sec:nsgrap}.
    
    We start by fixing some notation. For the rest of this section, we assume $g$, as used in the Gabor system $\mathcal{G}(g,a,b)$, to be compactly supported with $\supp(g) = [c,d]$.
    We write 
    \begin{equation}\label{eq:suppDUgab}
      \begin{split}
      I_{n,0} & = [c,d]+na,\\ 
      I^+_{n,k} & = [c-(k-1)a+kb^{-1},d]+na,\\
      I^-_{n,k} & = [c,d+(k-1)a-kb^{-1}]+na,
      \end{split}
    \end{equation}
    for all $k\in\NN,n\in\ZZ$. 
    
    These sets will be helpful in describing both the support of the weight functions of the Walnut-like representation of $\bd{S}^{-1}$, as well as the support of the canonical dual
    window $\bd{S}^{-1}g$ in the case that $\mathcal{G}(g,a,b)$ constitutes a frame. The conditions placed on $\mathcal{G}(g,a,b)$ in Theorems \ref{thm:chkiki} and \ref{thm:regcas} will be seen to imply 
    $I^\pm_{n,k+1}\subseteq I^\pm_{n,k}\subseteq I_{n,0}$ and $I^+_{n,1} \cap I^-_{n,1} = \emptyset$ for all $n\in\ZZ,k\in\NN$.\\
    
    The following theorem combines two results in~\cite{chkiki10} rewritten in our notation:

    \begin{Thm}[Christensen, Kim, Kim]\label{thm:chkiki} Let $g\in\LtR$ supported on $[-1,1]$ and $b\in ]1/2,1[$. Assume that $\mathcal{G}(g,1,b)$ is a frame and set $K: = \lfloor\frac{b}{1-b}\rfloor$, then 
      \begin{enumerate}
      \item[(i)] \cite[Th 2.1]{chkiki10} there exists a dual window $h\in\LtR$ with $\supp{h}\subseteq [-K,K]$.
      \item[(ii)] \cite[Lem 3.2]{chkiki10} If $g$ is bounded, $K > 1$ and $\mathcal{G}(h,1,b)$, with $h\in\LtR$ supported on $[-K,K]$, is a dual frame, then 
	$h$ is essentially supported on a subset of 
	\begin{equation*}
	  I_{0,0} \cup \bigcup\limits_{k=1}^{K} I^-_{-k,k} \cup I^+_{k,k}.
	\end{equation*}
      \end{enumerate}
    \end{Thm}

    The main tool used in \cite{chkiki10} is the duality condition for Gabor Bessel sequences $\mathcal{G}(g,a,b)$ and $\mathcal{G}(h,a,b)$ to form dual frames~\cite{rosh95,rosh97,ja98}
    \begin{equation}\label{eq:roshe}
      b^{-1}\sum_{n\in\ZZ} T_{na}h\overline{T_{kb^{-1}+na}g} = \begin{cases}
								  1\ a.e. & \text{for } k = 0,\\
								  0\ a.e. & \text{else.}
								\end{cases}
    \end{equation}
    In Section \ref{ssec:duality}, we will discuss the existence of a similar duality condition for nonstationary Gabor systems $\mathcal{G}(\bd{g},\bd{b})$ and $\mathcal{G}(\bd{h},\bd{b})$.
    
    Our following result is a restriction of Theorem \ref{thm:mainres} to Gabor systems, showing that the canonical dual window of $\mathcal{G}(\bd{g},\bd{b})$ satisfies the properties attributed
    to $h$ in Theorem \ref{thm:chkiki} (i) and (ii). The conditions on $g,\ a$ and $b$, while written differently as a preparation for Theorem \ref{thm:mainres}, are equivalent to those in Theorem \ref{thm:chkiki}. 
    We note that, by restricting $g$ to be a continuous, compactly supported function
    with finitely many zeros inside its support, Christensen, Kim and Kim show that the frame property of $\mathcal{G}(g,1,b)$ is equivalent to the existence of a continuous function
    $h\in\LtR$, with support contained in $[-K,K]$ such that $g,h$ satisfy the duality relations above. Our result investigates the structure of the inverse frame operator and derives properties
    of the canonical dual frame, but we do not attempt to characterize the frame property.

    \begin{Thm}\label{thm:regcas} Let $g\in\LtR$ with $\supp(g)\subseteq [c,d]$ and $d > c$. Furthermore $a\in [\frac{d-c}{2},d-c[$, $b\in ]0,\frac{1}{a}[$ and $K = \lfloor\frac{(d-c-a)b}{1-ab}\rfloor$. If $\mathcal{G}(g,a,b)$ is a frame, the following hold.
      \begin{enumerate}
      \item[(i)] The inverse frame operator $\bd{S}^{-1}$ has a Walnut-like representation of the form 
	\begin{equation}
	  \bd{S}^{-1}f = \sum_{k=0}^K \omega_k \bd{T}_{-kb^{-1}}f,
	\end{equation}
	with $\supp(\omega_k) \subseteq \bigcup_{n\in\ZZ} I^{-\operatorname{sgn}(k)}_{n,|k|}$ for all $k\neq 0$.
      \item[(ii)] The canonical dual window $\tilde{g} = \bd{S}^{-1}g\in\LtR$ satisfies
	  \begin{equation}
	    \supp(\tilde{g}) \subseteq I_{0,0} \cup \bigcup_{k=1}^{K} I^-_{-k,k} \cup I^+_{k,k}.% \cup I_X,
	  \end{equation}
      \end{enumerate}
    \end{Thm}

    Borrowing intuition from the discrete case, the inverse frame operator $\bd{S}^{-1}$ can, according to Theorem \ref{thm:regcas} (i), informally be interpreted as an infinitesimal matrix, supported only on the main diagonal and a discrete set of side-diagonals which in turn are non-zero only on specific intervals. For an illustration, see Figure \ref{fig:gab_frameops}.
  
  \begin{figure}[t!]
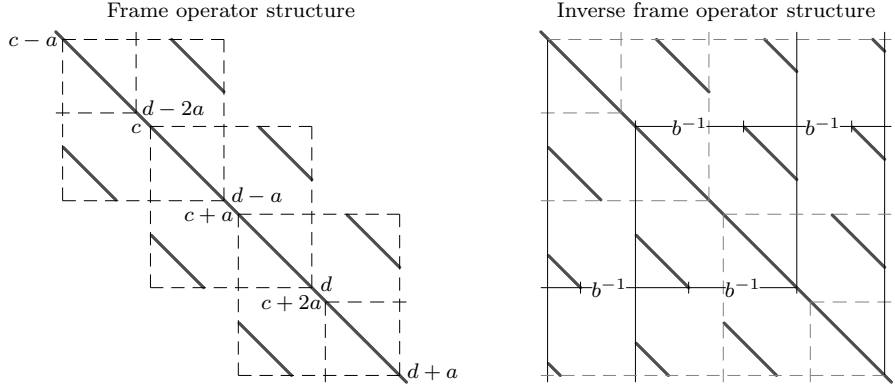

  \begin{center}
\providelength{\AxesLineWidth}       \setlength{\AxesLineWidth}{0.3pt}
\providelength{\plotwidth}           \setlength{\plotwidth}{8cm} % width of the axes only
\providelength{\LineWidth}           \setlength{\LineWidth}{0.4pt}
\providelength{\MarkerSize}          \setlength{\MarkerSize}{1.2pt}
\newrgbcolor{GridColor}{0.8 0.8 0.8}

% Begin Figure:-------------------------------------------
\psset{xunit=0.0014\plotwidth,yunit=0.0014\plotwidth}
%\psset{xunit=0.002398\plotwidth,yunit=0.002398\plotwidth}
\vspace{.3cm}\hspace{17pt}
%
\vspace{.1cm}
    \caption{Section of a Gabor frame operator and its inverse in the setting of Theorem \ref{thm:regcas} (schematic illustration). \emph{Left}: The weights correspond to side-diagonal entries of a matrix, with $\omega_0$ the main diagonal and $\omega_{\pm 1}$ located on side-diagonal $\pm b^{-1}$. Grey diagonal lines indicate non-zero entries in the side-diagonals/weights and we see that at most $3$ entries in each row are non-zero. Dashed lines indicate the support of the translates of $g$. \emph{Right}: The inverse frame operator additionally possesses a regularly spaced set of weights $\omega_{k}$ located on the side-diagonals $kb^{-1}$. Their non-zero entries are constrained by the support of the respective translates of $g$, indicated by horizontal and vertical lines. The parameter choice leads to shrinking support for weights located further from the main diagonal.}\label{fig:gab_frameops}
  \end{center}
  \end{figure}
    
    We see that Theorems \ref{thm:chkiki} and \ref{thm:regcas} are complementary and shed light on the same problem from somewhat different points of view.

    \begin{Exa}\label{ex:gabor1}
      Assume that $\mathcal{G}(g,\frac{7}{6},\frac{3}{5})$, with $g\in\LtR$ continuous, $\supp(g) = [-1,1]$ and $g(x)>0$ for all $x\in]-1,1[$, constitutes a frame. Then 
      \begin{equation*}
	  \bd{S}^{-1}f = \sum_{k\in\{-1,0,1\}} \omega_k \bd{T}_{-kb^{-1}}f,
      \end{equation*}
      with the essential supports of $\omega_1$ and $\omega_{-1}$ contained in $\bigcup_{n\in\ZZ} \left[-1+\frac{7n}{6},1+\frac{7n-10}{6}\right]$ and $\bigcup_{n\in\ZZ} \left[-1+\frac{7n+10}{6},1+\frac{7n}{6}\right]$, respectively.
      Consequently,
      \begin{equation*} 
	\supp(\bd{S}^{-1}g) \subseteq \left[-\frac{13}{6},-\frac{11}{6}\right] \cup \left[-1,1\right] \cup \left[\frac{11}{6},\frac{13}{6}\right],
      \end{equation*}
      since $I^+_{n,k} = I^-_{n,k} = \emptyset$, for all $k>1$.
    \end{Exa}

    This example raises the question when $\omega_k \equiv 0$ for $|k|>1$ can be guaranteed, i.e. the weights associated with $\bd{S}^{-1}$ are supported on the same set as those associated with $\bd{S}$.
    An answer is given in the following Corollary.

    \begin{Cor}
      Let $\mathcal{G}(g,a,b)$ as in Theorem \ref{thm:regcas}, with $b\in ]0,\frac{2}{d-c+a}[$, then 
      \begin{equation*}
	  \bd{S}^{-1}f = \sum_{k\in\{-1,0,1\}} \omega_k \bd{T}_{-kb^{-1}}f,
      \end{equation*}
      and $\omega_{1} \equiv 0$ outside $\bigcup_{n\in\ZZ} I^-_{n,1}$, $\omega_{-1} \equiv 0$ outside $\bigcup_{n\in\ZZ} I^+_{n,1}$. The same support conditions hold for $\bd{S}$, albeit with different weight functions.
    \end{Cor}
    \begin{proof}
	To obtain the statement for $\bd{S}^{-1}$, apply Theorem \ref{thm:regcas}(i) and simply check that $I^+_{n,k} = I^-_{n,k} = \emptyset$, for all $k>1$. For $\bd{S}$, the statement follows by applying the conditions of Theorem \ref{thm:regcas} to the Walnut representation \eqref{eq:walnut}.
    \end{proof}
    
    Under the conditions above, it is reasonable to assume that it is possible to find a dual window with support in $[c,d]$. As can be shown by applying the duality condition \eqref{eq:roshe}, this is true in many cases. 

    \begin{Cor}\label{thm:shsuppgab}
      Let $\mathcal{G}(g,a,b)$ with $g\in\LtR$ and $\supp(g)\subseteq [c,d]$ be a Gabor Bessel sequence as in Theorem \ref{thm:regcas} with $b \in ]0,\frac{2}{d-c+a}[$.
      %but not necessarily having a lower frame bound.  
      \begin{itemize}
      \item[(a)] Let $\mathcal{G}(h,a,b)$ a Gabor Bessel sequence with $h\in\LtR$, $\supp(h) \subseteq [c,d]$. $\mathcal{G}(g,a,b)$ and $\mathcal{G}(h,a,b)$ are dual frames if and only if the following hold:
	\begin{itemize}
      \item For almost every $x\in [c,c+a[$:
		  \begin{equation}\label{eq:dualityagab}
		    h\overline{g} + \bd{T}_{-a}(h\overline{g}) = b,\tag{a.i - Gabor}
		  \end{equation}
      \item For almost every $x\in I^-_{0,1}$:
		  \begin{equation}\label{eq:dualitybgab}
		    h\bd{T}_{-b^{-1}}\overline{g} = 0  \tag{a.ii - Gabor}
		  \end{equation}
      \item For almost every $x\in I^+_{0,1}$:
		  \begin{equation}\label{eq:dualitycgab}
		    h\bd{T}_{b^{-1}}\overline{g} = 0  \tag{a.iii - Gabor}
		  \end{equation}
	\end{itemize}
	\item [(b)] A Bessel sequence $\mathcal{G}(h,a,b)$ with $h\in\LtR$ and $\supp(h) \subseteq [c,d]$ exists, such that the pair $\mathcal{G}(g,a,b)$, $\mathcal{G}(h,a,b)$ satisfy (a), if and only if there is some $A>0$ such that the following hold:
      \begin{equation}\label{eq:exist1gab}\tag{b.i}
	|g(t)| \geq A\ \text{ or }\ |\bd{T}_{-a}g(t)| \geq A \text{ for a.e. } t\in [c,c+a[,
      \end{equation}
      \begin{equation}\label{eq:exist2gab}\tag{b.ii}
	|\bd{T}_{-a}g| \geq A \text{ a.e. on } \supp(\bd{T}_{-b^{-1}}g)\cap I^-_{0,1}
      \end{equation}
      and
      \begin{equation}\label{eq:exist3gab}\tag{b.iii}
	|\bd{T}_{a}g| \geq A \text{ a.e. on } \supp(\bd{T}_{b^{-1}}g)\cap I^+_{0,1}.
      \end{equation}
      \end{itemize}
    \end{Cor}

    Note that any real, continuous $g$ with $0<g(x)<1$ for all $x\in (c,d)$ satisfies Corollary \ref{thm:shsuppgab}(b) for all $a<d-c$. Furthermore, under the assumptions $0<g(x)<1$ for all $x\in (c,d)$
    any $h\in\LtR$ with $\supp(h)\subseteq [c,d]$, satisfying  Corollary \ref{thm:shsuppgab}(a) must have its essential support contained in $[d-b^{-1},c+b^{-1}]$. The result above, the restriction of Corollary \ref{thm:shsupp} to the regular Gabor case, is little more than a reduction of the duality condition \eqref{eq:roshe} to systems $\mathcal{G}(g,a,b)$ with $\supp(g)\subseteq [c,d]$ and $b \in ]0,\frac{2}{d-c+a}[$. 
    We see that pairs of dual frames with small support can be found if the painless case conditions are almost fulfilled.
    
    A more general result, improving the support condition in Theorem \ref{thm:chkiki}, can be found in \cite{chkiki12}. It cannot, however, easily be generalized to nonstationary Gabor frames.    

%%% ---------------------------------------------------------------------------------------------------------------
%%% -----------------------------------------Nonstationary---------------------------------------------------------    
%%% ---------------------------------------------------------------------------------------------------------------

  \section{Nonstationary Gabor frames}\label{sec:nsgrap}
  
  We now generalize the notation used in Section \ref{sec:regcas} to the nonstationary setting and state our results in the general case. Since the modulation parameters $b_n$ need not be equal anymore, we will work with
  \begin{equation}
    \begin{split}
    B^+_{n,k} &:= \sum_{j=0}^{k-1} b^{-1}_{n+j}\ \text{and }\\ 
    B^-_{n,k} &:= \sum_{j=0}^{k-1} b^{-1}_{n-j},\ \forall~ n\in\ZZ, k\in\NN.
    \end{split}
  \end{equation}
  Then, for the nonstationary Gabor system $\mathcal{G}(\bd{g},\bd{b})$, with $\supp (g_n) = [c_n,d_n]$ for all $n\in\ZZ$, we set
  \begin{equation}\label{eq:suppDU}
    \begin{split}
    I_{n,0} &= [c_n,d_n],\\ 
    I^+_{n,k} &= [c_{n-k+1}+B^-_{n,k},d_n],\\ 
    I^-_{n,k} & = [c_n,d_{n+k-1}-B^+_{n,k}]
    \end{split}
  \end{equation}
  for $n,k$ as before. Note that for $g_n = \bd{T}_{na}g$ and $b_n = b$ for all $n\in\ZZ$, these sets coincide with those in the previous section.
  
  As before, the notational conventions above will be helpful in describing the structure inherent to the Walnut-like representation of inverse nonstationary Gabor frame operators.    
  The conditions on $\mathcal{G}(\bd{g},\bd{b})$ in Theorem \ref{thm:mainres} below imply $I^\pm_{n,k+1}\subseteq I^\pm_{n,k}\subseteq I_{n,0}$ and $I^+_{n,1} \cap I^-_{n,1} = \emptyset$ for all $n\in\ZZ,k\in\NN$. Some intuition can be gained from likening the NSG frame operator and its inverse to a sparse, infinitesimal matrix with a structured set on non-zero side-diagonals that are, in turn, non-zero only on specific intervals. For an illustration, see Figure \ref{fig:nsg_frameops}.

    \begin{figure}[t!]
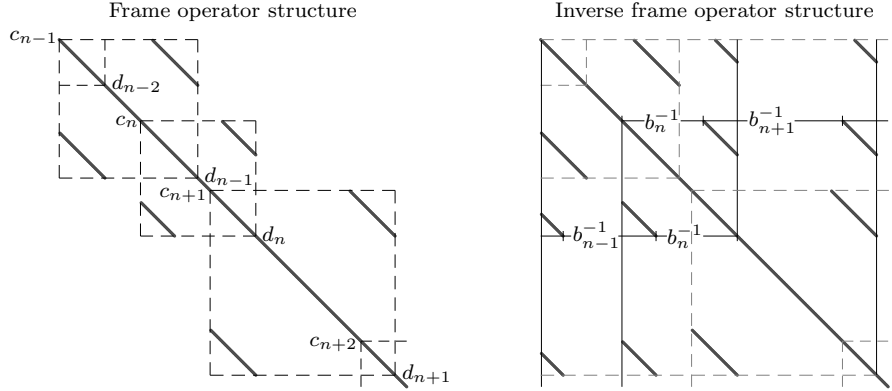

  \begin{center}
\providelength{\AxesLineWidth}       \setlength{\AxesLineWidth}{0.3pt}
\providelength{\plotwidth}           \setlength{\plotwidth}{8cm} % width of the axes only
\providelength{\LineWidth}           \setlength{\LineWidth}{0.4pt}
\providelength{\MarkerSize}          \setlength{\MarkerSize}{1.2pt}
\newrgbcolor{GridColor}{0.8 0.8 0.8}

% Begin Figure:-------------------------------------------
\psset{xunit=0.0012\plotwidth,yunit=0.0012\plotwidth}
%\psset{xunit=0.002079\plotwidth,yunit=0.002079\plotwidth}
\vspace{.3cm}\hspace{17pt}
%

    \caption{Section of a NSG frame operator and its inverse in the setting of Theorem \ref{thm:mainres} (schematic illustration). \emph{Left}: The weights correspond to side-diagonal entries of a matrix, with $\omega_0$ the main diagonal and $\omega_{n,\pm 1}$ located on side-diagonal $\pm b_n^{-1}$. Grey diagonal lines indicate non-zero entries in the side-diagonals/weights and we see that at most $3$ entries in each row are non-zero. Dashed lines indicate the support of the individual window functions. \emph{Right}: The inverse frame operator additionally possesses an irregularly spaced set of weights $\omega_{n,k}$ located on the side-diagonals $\sgn(k)B^{\sgn(k)}_{n,|k|}$, i.e. dependent on the non-uniform frequency steps $b_n$. Their non-zero entries are constrained by the support of the respective windows, indicated by horizontal and vertical lines. The parameter choice leads to shrinking support for weights located further from the main diagonal.}\label{fig:nsg_frameops}
  \end{center}
  \end{figure}

  The following theorem details the structure of the inverse frame operator and the canonical dual frame $(\widetilde{g_{m,n}})_{m,n}$:
  
  \begin{Thm}\label{thm:mainres}
  Let $\mathcal{G}(\bd{g},\bd{b})$ be a nonstationary Gabor frame with $g_n\in\LtR$, $\supp(g_n) = [c_n,d_n]$, $c_n<d_n$, and $b_n \in ]\frac{1}{d_n-c_n},\infty[$ for all $n\in\ZZ$.
  If $\epsilon > 0$ exists such that $d_{n-1} \leq c_{n+1}$ and $b_n^{-1}\geq\max\{\frac{d_n-c_n}{2},c_{n+1}-c_n,d_n-d_{n-1}\}+\epsilon$ for all $n\in\ZZ$, then the inverse frame operator $\bd{S}_{\bd{g},\bd{b}}^{-1} = \bd{S}^{-1}$ has a Walnut-like representation of the form 
  \begin{itemize}
      \item[(i)]  
	\begin{equation}
	  \bd{S}^{-1}f =  \omega_0 f + \sum_{n\in\ZZ}\sum_{k\in\ZZ\setminus\{0\}} \omega_{n,k} \bd{T}_{-\sgn(k)B^{\sgn(k)}_{n,|k|}}f,\ \forall~f\in\LtR
	\end{equation}
	where $\supp(\omega_{n,k}) \subseteq I^{-\sgn(k)}_{n,|k|}$ for all $n\in\ZZ,k\in\ZZ\setminus\{0\}$.
      \item[(ii)] for any fixed $n \in\NN$ the elements $\{\widetilde{g_{m,n}} = \bd{S}^{-1}g_{m,n}\}_{m\in\ZZ}$ of the canonical dual frame satisfy
	\begin{equation}
	    \supp(\widetilde{g_{m,n}}) \subseteq I_{n,0} \cup \bigcup\limits_{k\in\NN} I^-_{n-k,k} \cup I^+_{n+k,k}.% \cup I_X,
	\end{equation}
	
      \item[(iii)] The elements $\widetilde{g_{m,n}},\ m\neq 0$ of the canonical dual frame can be derived from $\widetilde{g_{0,n}}$ by
	\begin{align}
	    \widetilde{g_{m,n}} & = \bd{M}_{mb_n}\Big(\widetilde{g_{0,n}}|_{I^{(0)}_n} + \sum_{k\in\NN} \widetilde{g_{0,n}}|_{I^-_{n-k,k}}\exp(2\pi i mb_nB^+_{n-k,k}) \nonumber \\
	    & \hspace{85pt} + \widetilde{g_{0,n}}|_{I^+_{n+k,k}}\exp(-2\pi i mb_nB^-_{n+k,k})\Big).
	\end{align}
      \item[(iv)] For each $n\in\ZZ$, there exists $k_n\in\NN$ such that $I^\pm_{n,k} = \emptyset$
	      for all $k \geq k_n$. Furthermore, if a constant $C<\infty$ exists, such that $\max_n(d_n-c_n-b_n^{-1}) \leq C$, then $I^\pm_{n,k} = \emptyset$
	      for all $n\in\ZZ$ and $\NN \ni k \geq C/\epsilon$.
  \end{itemize}
  \end{Thm}

  Loosely speaking, the above theorem can be read as follows: Whenever a nonstationary Gabor system, comprised of compactly supported window functions with moderate overlap and sufficiently small modulation 
  parameters, constitutes a frame, then 
  \begin{enumerate}
  \item[(i)] the inverse frame operator possesses a Walnut-like representation with compactly supported off-diagonal weight functions
  \item[(ii)] each element of the canonical dual frame is supported on a finite, disjoint union of compact intervals.
  \item[(iii)] the canonical dual frame of $\mathcal{G}(\bd{g},\bd{b})$ is ``almost'' a nonstationary Gabor system with the same modulation parameters. Some phase shifts may occur, though.
  \item[(iv)] for fixed $n\in\ZZ$ only finitely many of the intervals $I^\pm_{n,k}$ are non-empty. If the window sizes behave nicely, there is a uniform bound on the number of non-empty sets, 
	      valid for all $n\in\ZZ$.
  \end{enumerate}
  
  It is imminent from Theorem \ref{thm:mainres}(ii) and (iii), that we can only guarantee the canonical dual system to be a NSG system with the same modulation parameters, if either $I^\pm_{n,k} = \emptyset$ for all 
  $n,k$ or $b_n = b$ for all $n\in\ZZ$. While other constructions are conceivable, e.g. using uniform modulation parameters in a blockwise fashion, separated by a window $g_n$ with $b_n^{-1} \geq d_n-c_n$, they 
  require great care in the choice of both window functions and parameters. More intuitive constructions such as the choice of a uniform undersampling factor, i.e. $b_n^{-1} = \alpha(d_n-c_n)$ for some $\alpha < 1$,
  do not leave the structure of the original system intact.
  
  To recover Theorem \ref{thm:regcas} from Theorem \ref{thm:mainres}, combine (i) and (iv); noting that $B^\pm_{n,k} = kb^{-1}$ for all $n,k$, take $\omega_k = \sum_{n} \omega_{n,k}$ for $k\neq 0$.
  
  A closer look at the intervals $I^\pm_{n,k}$ shows that, under the conditions of Theorem \ref{thm:mainres}, $I^\pm_{n,k+1} \subsetneq I^\pm_{n,k}$ is guaranteed. That is, the sets $I^\pm_{n,k}$ are 
  strictly shrinking for $n$ fixed and $k$ increasing. Lemma \ref{lem:intervals} in Section \ref{sec:prodis} takes a closer look at how these sets are intertwined.

  As in the regular Gabor case, it is reasonable to ask whether the weights of the inverse frame operator are supported on the same set as those of the original frame operator. 

  \begin{Cor}\label{cor:samesupp}
      Let $\mathcal{G}(\bd{g},\bd{b})$ be as in Theorem \ref{thm:mainres}, with $b_n \in \left]0,\frac{b_{n-1}}{b_{n-1}(d_n-c_{n-1})-1}\right[$ for all $n\in\ZZ$, then 
	    \begin{equation*}
	  \bd{S}^{-1}f = \omega_0f + \sum_{n\in\ZZ}\sum_{k\in\{-1,1\}} \omega_{n,k} \bd{T}_{-\sgn(k)B^{\sgn(k)}_{n,|k|}}f,
      \end{equation*}
      and $\supp(\omega_{n,\pm 1}) \subseteq I^{\mp}_{n,1}$. The same support conditions hold for $\bd{S}$, albeit with different weights.
  \end{Cor}
  \begin{proof}
	Apply Theorem \ref{thm:mainres}(i) and simply check that $I^+_{n,k} = I^-_{n,k} = \emptyset$, for all $n\in\ZZ,k>1$ to show the statement for $\bd{S}^{-1}$. For $\bd{S}$, apply the Walnut representation \eqref{eq:walnut} and check the conditions of Theorem \ref{thm:mainres}.
  \end{proof}
  
  So far, we have investigated the structure of inverse NSG frame operators and the canonical dual frames of NSG systems. We have seen that only few particular choices of $\mathcal{G}(\bd{g},\bd{b})$ yield a canonical dual frame of the form $\mathcal{G}(\tilde{\bd{g}},\bd{b})$. Yet, this does not exclude the existence of a dual system $\mathcal{G}(\bd{h},\bd{b})$ per se. To further illuminate this problem, 
  we will deduce a sufficient, and in many standard cases necessary, condition for duality of two NSG systems $\mathcal{G}(\bd{g},\bd{b})$ and $\mathcal{G}(\bd{h},\bd{b})$ to constitute dual frames. As an illustrative example, we will apply the result in the setting of Corollary \ref{cor:samesupp}.
  
  \subsection{Towards a duality condition}\label{ssec:duality}
  
  The Walnut representation \eqref{eq:walnut} is an efficient way to describe the action of a NSG frame-type operator. However, to determine duality of two NSG systems, it is beneficial to rearrange the summations ordered by the appearing translate of $f$. More precisely, define for any sequence $\bd{b} = (b_n\in\RR^+)_n$ the countable set $E_{\bd{b}}$ by
  \begin{equation}
    E_{\bd{b}} = \{ x\in\RR~:~\exists~(m,n)\in\ZZ^2 \text{ s.t. } x = mb^{-1}_n \}.
  \end{equation}
  
  Furthermore, to prevent pathologies, we introduce the following notion of ``nice'' nonstationary Gabor systems.
  \begin{Def}\label{def:wellbhvd}
    We call a nonstationary Gabor system $\mathcal{G}(\bd{g},\bd{b})$ \emph{well-behaved}, if either of the following holds:
    \begin{enumerate}
    \item[(i)] $E_{\bd{b}}$ is free of accumulation points.%is a $\delta$-separated set, i.e. $|x-y|\geq \delta$ for all $x,y\in E_{\bd{b}}$ with $x\neq y$.
    \item[(ii)] For all $n\in\ZZ$, $g_n$ is compactly supported on some interval $[c_n,d_n]$ and $O_n := \{ l\in\ZZ ~:~ c_l < d_n \text{ and } d_l > c_n \}$ is finite.
    \end{enumerate}
  \end{Def}
  
  The flexibility gained by the way a NSG system is defined allows the construction of a multitude of pathological cases that are generally not interesting for practical purposes. Note that the functions $g_n$ and $h_n$ are usually desired to be well concentrated in time and frequency. Further, they should be evenly distributed over time. Consequently, only finitely many compactly supported windows overlapping is a rather weak restriction. On the other hand $\{b_n^{-1}~:~n\in\ZZ\}$ being $\delta$-separated, i.e. either $b_n^{-1} = b_l^{-1}$ or $|b_n^{-1}-b_l^{-1}|\geq \delta$ for all $n,l\in\ZZ$ is enough to guarantee $E_{\bd{b}}$ being free of accumulation points. 
  
  We can now formulate an alternative version of the Walnut representation \eqref{eq:walnut}, valid on a dense subspace of $\LtR$.
  
  \begin{Cor}\label{cor:walnutalt}
    Let $\mathcal{G}(\bd{g},\bd{b})$ and $\mathcal{G}(\bd{h},\bd{b})$ be well-behaved nonstationary Gabor Bessel sequences with $b_n\in\RR^+$ and $g_n,h_n\in\LtR$, for all $n\in\ZZ$. Then, for all $f\in\LtR$ with compact support,% the associated frame-type operator $\bd{S}_{\bd{g},\bd{h},\bd{b}}$,% admits a Walnut-like representation of the form
    \begin{equation}\label{eq:walalt}
      \bd{S}_{\bd{g},\bd{h},\bd{b}}f = \sum_{x\in E_{\bd{b}}} \omega_x \bd{T}_x f,
    \end{equation}
    with 
    \begin{equation}\label{eq:walaltweights}
      \omega_0 = \sum_{n\in\ZZ} b_n^{-1} h_n\overline{g_n}\text{ and } \omega_x = \sum_{\substack{(m,n)\in\ZZ^2 \\ mb^{-1}_n = x}} b_n^{-1} h_n \bd{T}_{x}\overline{g_n}\text{ for } x\neq 0.
    \end{equation}
    Moreover, the sum in \eqref{eq:walalt} is absolutely convergent. Consequently, the extension to $\LtR$ of the bounded, linear operator defined by the right-hand side of \eqref{eq:walalt} equals $\bd{S}_{\bd{g},\bd{h},\bd{b}}$.
  \end{Cor}
  \begin{proof}
    By Proposition \ref{pro:diagsbound}, $\sum_{n\in\ZZ} b^{-1}_n |h_n\bd{T}_x \overline{g_n}| \leq B$ almost everywhere, for any $x\in\RR$. Now let $I$ be any finite interval such that $\supp(f)+\supp(f) \subseteq I$. If $\mathcal{G}(\bd{g},\bd{b})$, $\mathcal{G}(\bd{h},\bd{b})$ are well-behaved in the sense of Definition \ref{def:wellbhvd}(i), then $E_{\bd{b}}\cap I$ is a finite set and
    \begin{align}\label{eq:finitesum}
      \lefteqn{\Big| \sum_{x\in E_{\bd{b}}} \sum_{\substack{ (n,k)\in\ZZ^2 \\ x = kb_n^{-1}}} b_n^{-1}(h_n\bd{T}_x\overline{g_n})\bd{T}_x f \Big|} \nonumber \\
      & \leq \sum_{x\in E_{\bd{b}}\cap I} \sum_{\substack{ (n,k)\in\ZZ^2 \\ x = kb_n^{-1}}} b_n^{-1}|h_n\bd{T}_x\overline{g_n}||\bd{T}_x f| \nonumber \\
      & \leq B \sum_{x\in E_{\bd{b}}\cap I} |\bd{T}_x f| < \infty\ \text{a.e. on } I,
    \end{align}
    with absolute convergence. If on the other hand, $\mathcal{G}(\bd{g},\bd{b})$, $\mathcal{G}(\bd{h},\bd{b})$ are well-behaved in the sense of Definition \ref{def:wellbhvd}(ii), then the sum over $E_{\bd{b}}$ is locally finite. Thus, by the Walnut representation \eqref{eq:walnut} of $\bd{S}_{\bd{g},\bd{h},\bd{b}}$:
    \begin{equation}
      \sum_{x\in E_{\bd{b}}} \omega_x \bd{T}_x f = \sum_{n,k\in\ZZ} b_n^{-1} (h_n \bd{T}_{kb^{-1}_n}\overline{g_n}) \bd{T}_{kb^{-1}_n} f = \bd{S}_{\bd{g},\bd{h},\bd{b}} f,\ \forall~f\in\LtR.
    \end{equation}
    Since $\sum_{x\in E_{\bd{b}}} \omega_x \bd{T}_x = \bd{S}_{\bd{g},\bd{h},\bd{b}}$ on a dense subspace of $\LtR$, the extension of $\sum_{x\in E_{\bd{b}}} \omega_x \bd{T}_x$ to $\LtR$ equals
    $\bd{S}_{\bd{g},\bd{h},\bd{b}}$.
  \end{proof}
  
  It is easy to see that $\omega_0 \equiv 1$ and $\omega_x \equiv 0$ for $x\neq 0$ is a sufficient condition for $\bd{S}_{\bd{g},\bd{h},\bd{b}}f = f$ and thus for $\mathcal{G}(\bd{g},\bd{b})$ and $\mathcal{G}(\bd{h},\bd{b})$ to be dual frames. For well-behaved systems $\mathcal{G}(\bd{g},\bd{b})$ and $\mathcal{G}(\bd{h},\bd{b})$, it can be shown to be necessary as well. 

  \begin{Thm}\label{thm:dualitycond} 
    Let $\mathcal{G}(\bd{g},\bd{b})$, $\mathcal{G}(\bd{h},\bd{b})$ be nonstationary Gabor Bessel sequences with $g_n, h_n \in\LtR$, $b_n\in\RR^+$. Then 
    \begin{equation}\label{eq:dualitycond}
      \sum_{n\in\ZZ} b_n^{-1} h_n\overline{g_n} \equiv 1 \text{ and } \sum_{\substack{(m,n)\in\ZZ^2 \\ mb^{-1}_n = x}} b_n^{-1} h_n \bd{T}_{x}\overline{g_n} \equiv 0 \text{ for } x\neq 0
    \end{equation}
    implies duality in the sense that 
    \begin{equation}\label{eq:classicduality}
      \langle f_1,f_2 \rangle = \sum_{m,n\in\ZZ} \langle f_1,g_{m,n} \rangle \langle h_{m,n},f_2\rangle \text{ for all } f_1,f_2\in\LtR.
    \end{equation}
    If furthermore both $\mathcal{G}(\bd{g},\bd{b})$ and $\mathcal{G}(\bd{h},\bd{b})$ are well-behaved, the converse holds as well.
  \end{Thm}
  \begin{proof}
    Using the Bessel property of $\mathcal{G}(\bd{g},\bd{b})$ and $\mathcal{G}(\bd{h},\bd{b})$, we can interchange summation and integration in the right-hand side of Equation \eqref{eq:classicduality} arriving at
    \begin{equation*}
      \sum_{m,n\in\ZZ} \langle f_1,g_{m,n} \rangle \langle h_{m,n},f_2\rangle = \left\langle \sum_{m,n\in\ZZ} \langle f_1,g_{m,n} \rangle h_{m,n},f_2\right\rangle.
    \end{equation*}
    Assume Equation \eqref{eq:dualitycond} holds. Invoking the alternate Walnut representation \eqref{eq:walalt} of $\bd{S}_{\bd{g},\bd{h},\bd{b}}$, we see that for all compactly supported $f_1$,
    \begin{align*}
      \sum_{m,n\in\ZZ} \langle f_1,g_{m,n} \rangle h_{m,n} & = \bd{S}_{\bd{g},\bd{h},\bd{b}}f_1 \\
      & = \sum_{x\in E_{\bd{b}}} \omega_x \bd{T}_x f_1 \\
      & = \omega_0 \bd{T}_0 f_1 = f_1.
    \end{align*}    
    Therefore, Equation \eqref{eq:classicduality} holds for all compactly supported $f_1\in\LtR$ and by density for all $f_1\in\LtR$, proving the first inference.
    We prove the converse inference by contradiction, assuming \eqref{eq:dualitycond} to be violated, then provide a counterexample to \eqref{eq:classicduality}. Let $\mathcal{G}(\bd{g},\bd{b})$, $\mathcal{G}(\bd{h},\bd{b})$ be well-behaved in the sense of Definition \ref{def:wellbhvd}(i), i.e. if $\omega_x\neq 0$ for some $x\in\RR\setminus\{0\}$, we can choose $\delta > 0$ and $l\in\ZZ$, such that 
    $E_{\bd{b}} \cap B_{2\delta} (x) = \{x\}$ and $\omega_x|_{B_\delta (2l\delta)}\neq 0$. Let $f_1 = \chi_{B_\delta (2l\delta)-x}$ and $f_2 = \overline{\omega_x}\mid_{B_\delta (2l\delta)}$, then 
    \begin{equation*}
      0 \equiv \langle f_1,f_2 \rangle \neq \langle \bd{S}_{\bd{g},\bd{h},\bd{b}}f_1,f_2 \rangle  = \| \omega_x\mid_{B_\delta(2l\delta)} \|^2,
    \end{equation*}
    proving that $\omega_x \equiv 0$ for all $x\in\RR\setminus\{0\}$ is necessary. But then $\omega_0 \neq 1$ contradicting \eqref{eq:classicduality} can easily be seen. If instead $\mathcal{G}(\bd{g},\bd{b})$, $\mathcal{G}(\bd{h},\bd{b})$ are well-behaved in the sense of Definition \ref{def:wellbhvd}(ii), note that 
    \begin{equation}
      E_{\bd{b},n} = \{ x\in\RR, \exists~ m\in\ZZ,l\in O_n ~:~ x = mb^{-1}_l \}
    \end{equation}
    is free of accumulation points and apply the reasoning above.
  \end{proof}
  
  Note that duality of a pair of Bessel sequences in the sense of \eqref{eq:classicduality} implies the frame property for both involved Bessel sequences.
  
  \begin{Rem}
    For systems with uniform $\bd{b}$, i.e. $b_n = b$, the duality condition above reduces to the well-known conditions for Gabor frames ~\cite{rosh95,rosh97} or more generally, shift-invariant frames~\cite{ja98}. In both classical cases, the canonical dual frame inherits the structure of the original frame and thus the duality conditions are guaranteed to have a solution. This is not true for NSG systems in general. Indeed, we expect that for many choices of a NSG frame $\mathcal{G}(\bd{g},\bd{b})$, there is no system $\mathcal{G}(\bd{h},\bd{b})$ satisfying \eqref{eq:dualitycond}. 
  \end{Rem}
  
  \begin{Rem}
    The restriction to well-behaved NSG systems in Theorem \ref{thm:dualitycond} prevents us from recovering the equivalence of the duality conditions to the frame property for Wavelet systems, proven by Chui and Shi in \cite{chsh00}. However, the restriction to well-behaved systems allows for a straightforward proof, once all the ingredients are in place. Further relaxation of the conditions for well-behavedness, using the methods presented in \cite{chsh00}, is planned as future work.    
  \end{Rem}  
  
  Given a specific setup of $\bd{g}$ and $\bd{b}$, the duality conditions above may prove useful to determine the existence of a dual system that shares the modulation parameters $\bd{b}$. This is particularly interesting from an algorithmic point of view, since analysis and synthesis can be realized efficiently for NSG systems, but not for general frames. Here, we consider the setting of Corollary \ref{cor:samesupp} and show that dual pairs of NSG frames with compactly supported generators exist. We obtain the following result.

  \begin{Cor}\label{thm:shsupp}
      Let $\mathcal{G}(\bd{g},\bd{b})$ be a nonstationary Gabor Bessel sequence as in Theorem \ref{thm:mainres} with $g_n\in\LtR$, $b_n \in ]0,\frac{b_{n-1}}{b_{n-1}(d_n-c_{n-1})-1}[$ and $c_n\leq d_{n-1}$ for all $n\in\ZZ$.
      \begin{itemize}
      \item[(a)] Let $\mathcal{G}(h_n,b_n)$ a nonstationary Gabor Bessel sequence with $h_n\in\LtR$, $\supp(h_n) \subseteq [c_n,d_n]$. $\mathcal{G}(\bd{g},\bd{b})$ and $\mathcal{G}(\bd{h},\bd{b})$ are dual frames if and only if the following hold for all $n\in\ZZ$:
      \begin{itemize}
      \item For almost every $x\in [c_n,c_{n+1}[$:
		  \begin{equation}\label{eq:dualitya}
		    b_n^{-1}h_n\overline{g_n} + b_{n-1}^{-1}h_{n-1}\overline{g_{n-1}} = 1,\tag{a.i}
		  \end{equation}
      \item For almost every $x\in I^-_{n,1}$:
		  \begin{equation}\label{eq:dualityb}
		    h_n\bd{T}_{-b_n^{-1}}\overline{g_n} = 0  \tag{a.ii}
		  \end{equation}
      \item For almost every $x\in I^+_{n,1}$:
		  \begin{equation}\label{eq:dualityc}
		    h_n\bd{T}_{b_n^{-1}}\overline{g_n} = 0  \tag{a.iii}
		  \end{equation}     
      \end{itemize}
      \item [(b)] A Bessel sequence $\mathcal{G}(\bd{h},\bd{b})$ with $h_n\in\LtR$ and $\supp(h_n) \subseteq [c_n,d_n]$ for all $n\in\ZZ$ exists, such that the pair $\mathcal{G}(\bd{g},\bd{b})$, $\mathcal{G}(\bd{h},\bd{b})$ satisfy (a), if and only if there is some $A>0$ such that the following hold for all $n\in\ZZ$:
      \begin{equation}\label{eq:exist1}\tag{b.i}
	b_n^{-1/2}|g_n(t)| \geq A\ \text{ or }\ b_{n-1}^{-1/2}|g_{n-1}(t)| \geq A \text{ for a.e. } t\in[c_n,c_{n+1}[,
      \end{equation}
      \begin{equation}\label{eq:exist2}\tag{b.ii}
	b_{n-1}^{-1/2}|g_{n-1}| \geq A \text{ a.e. on } \supp(\bd{T}_{-b_n^{-1}}g_n)\cap I^-_{n,1}
      \end{equation}
      and
      \begin{equation}\label{eq:exist3}\tag{b.iii}
	b_{n+1}^{-1/2}|g_{n+1}| \geq A \text{ a.e. on } \supp(\bd{T}_{b_n^{-1}}g_n)\cap I^+_{n,1}.
      \end{equation}
      \end{itemize}
    \end{Cor}
    \begin{proof}
      The systems $\mathcal{G}(\bd{g},\bd{b})$ and $\mathcal{G}(\bd{h},\bd{b})$ are well-behaved in the sense of Definition \ref{def:wellbhvd}(ii). Thus they form a pair of dual nonstationary Gabor frames if and only if Equation \eqref{eq:dualitycond} is satisfied. Invoking the support conditions on the systems, we get
      \begin{equation}\label{eq:equal1}
	b_n^{-1}h_n\overline{g_n} + b_{n-1}^{-1}h_{n-1}\overline{g_{n-1}} = 1 \text{ a.e. on } [c_n,c_{n+1}[,\tag{a.i}
      \end{equation}
      \begin{equation}\label{eq:equal0a}
	h_n\bd{T}_{-b_n^{-1}}\overline{g_n} = 0 \text{ a.e. on } I^+_{n,0} \tag{a.ii}
      \end{equation}
      and
      \begin{equation}\label{eq:equal0b}
	h_n\bd{T}_{b_n^{-1}}\overline{g_n} = 0 \text{ a.e. on } I^-_{n,0}, \tag{a.iii}
      \end{equation}
      for all $n\in\ZZ$, concluding the proof of (a).
      We first prove that \eqref{eq:exist1} to \eqref{eq:exist3} are sufficient by constructing a dual Bessel sequence $\mathcal{G}(\bd{h},\bd{b})$ satisfying the support constraints. Let for all $n\in\ZZ$, $J^{0}_n$ be the largest open subset of $[c_n,c_{n+1}[$ such that $b_n^{-1/2}|g_n|\geq A$ almost everywhere on $J^{0}_n$ and $J^{1}_n = [c_{n+1},d_n]\setminus J^{0}_{n+1}$.
      Furthermore let us denote, for all $n\in\ZZ$, $J^-_n = I^-_{n,1} \cap \supp(T_{-b_n^{-1}}g_n)$ and $J^+_n = I^+_{n-1,1} \cap \supp(T_{b_{n-1}^{-1}}g_{n-1})$. Then
      \begin{equation*}
	h_n := \begin{cases}
		b_n/\overline{g_n}, & \text{on } \left(J^{0}_n\cup J^{1}_n\right) \setminus \left( J^+_n \cup J^-_n\right),\\
		0, & \text{ else },	         
	      \end{cases}
      \end{equation*}
      is well-defined almost everywhere for all $n\in\ZZ$. With this choice, it is easy to see that the conditions \eqref{eq:equal1} to \eqref{eq:equal0b} are satisfied. Furthermore $h_n(t) < 2\sqrt{b_n}/A$ almost everywhere. Thus $h_n\in L^\infty(\RR)$ and is compactly supported, in particular $h_n\in \LtR\cap W(L^\infty,\ell^1)$. We see that $\sum_n b_n^{-1}|h_n|^2 \leq 2/A^2$. Invoke the Walnut representation and apply the proof of Proposition \ref{pro:diagsbound} to see that $\mathcal{G}(\bd{h},\bd{b})$ is a Bessel sequence.
      
      For the converse, we assume either of \eqref{eq:exist1} to \eqref{eq:exist3} to be violated. Note that $h_n$ is uniquely determined almost everywhere on $J^-_{n+1}\cup J^+_{n-1}$. If \eqref{eq:exist2} or \eqref{eq:exist3} is violated, we can for every $\epsilon > 0$ find $n\in\ZZ$, such that $b_n^{-1/2}|g_n| < \epsilon$ and consequently $b_n^{-1/2}|h_n| > \epsilon^{-1}$ almost everywhere on a subset $M \subseteq  J^-_{n+1}\cup J^+_{n-1}$ of positive measure. Therefore, $\sum_n b_n^{-1}|h_n|^2 > \epsilon^{-2}$ on a set of positive measure, contradicting the Bessel condition by Proposition \ref{pro:diagsbound}. If on the other hand \eqref{eq:exist1} is violated, then we can for every $\epsilon > 0$ find $n\in\ZZ$, such that $b_n^{-1/2}|g_n| < \epsilon$ and $b_{n-1}^{-1/2}|g_{n-1}| < \epsilon$ almost everywhere on a subset $M \subseteq [c_n,c_{n+1}[$ of positive measure. Assume 
      \ref{eq:equal1} to be satisfied, i.e.
      \begin{equation*}
	1 = |b_n^{-1}h_n\overline{g_n} + b_{n-1}^{-1}h_{n-1}\overline{g_{n-1}} | \leq \epsilon \left( b_n^{-1/2}|h_n| + b_{n-1}^{-1/2}|h_{n-1}|\right)\ \text{a.e. on } M.
      \end{equation*}
      Then, almost everywhere on $M$, either $b_n^{-1/2}|h_n| > 1/2\epsilon$ or $b_{n-1}^{-1/2}|h_{n-1}| > 1/2\epsilon$, contradicting the Bessel condition by Proposition \ref{pro:diagsbound}.
    \end{proof}
    
    \begin{Rem}
      The proof shows that, given $g_n,g_{n-1},g_{n+1}$, $h_n$ is uniquely determined on $I^-_{n,1}\cup ]d_{n-1},c_{n+1}[ \cup I^+_{n,1}$, except for a zero set. Therefore, as Christensen, Kim and Kim have observed in the regular case~\cite{chkiki12}, the equation system \eqref{eq:equal1} to \eqref{eq:equal0b} is not solvable in general, if $I^-_{n,1}\cap I^+_{n-1,1} \neq \emptyset$ and for all $n\in\ZZ$, $\supp(h_n)\subseteq \supp(g_n)$. In the classical Gabor case, because $b_n = b$ for all $n\in\ZZ$, an appropriate increase of the size of $\supp(h_n)$ does the trick. We expect that this can be generalized to NSG systems with uniform $b_n$. In the general case however, non-uniformity of $b_n$ significantly complicates matters and further work is required to determine the solvability of \eqref{eq:equal1} to \eqref{eq:equal0b} even without support constraints on the $h_n$, i.e. whether any NSG system $\mathcal{G}(\bd{h},\bd{b})$, dual to $\mathcal{G}(\bd{g},\bd{b})$, can exist.
    \end{Rem}
    
    To recover Corollary \ref{thm:shsuppgab}, replace $g_n$ by $T_{na}g$ and $b_n$ by $b$ and observe the $a$-periodicity of Gabor systems. Note that $b<\frac{b}{b(d-c+a)-1}$ is equivalent to $b<\frac{2}{d-c+a}$.
    
%%% ---------------------------------------------------------------------------------------------------------------
%%% -----------------------------------------Proof and Discussion--------------------------------------------------    
%%% ---------------------------------------------------------------------------------------------------------------
    
  \section{Proof and discussion of Theorem \ref{thm:mainres}}\label{sec:prodis}

      Before we prove Theorem \ref{thm:mainres}, we collect some preliminary results about NSG systems $\mathcal{G}(\bd{g},\bd{b})$ satisfying the conditions of the theorem. First, $d_{n-1} \leq c_{n+1}$ guarantees that, except possibly at endpoints, at most two adjacent windows $g_n$ and $g_{n+1}$ overlap. Moreover, $b_n > \frac{1}{d_n-c_n}$ combined with $b_n^{-1} > c_{n+1} - c_n$ yields $c_{n+1} < c_n + b_n^{-1} < d_n$ and analogous, $c_n < d_n - b_n^{-1} < d_{n-1}$, implying that $g_n$, $g_{n+1}$ and $g_n$, $g_{n-1}$ overlap on a nontrivial interval. If $b_n \leq \frac{1}{d_n-c_n}$, then $c_n \leq d_{n-1}$ and $c_{n+1}\leq d_n$, are still necessary conditions for completeness of $\mathcal{G}(\bd{g},\bd{b})$. Further, $b_n < \frac{2}{d_n-c_n}$ yields $[c_n,d_n] \cap ([c_n,d_n]+kb_n^{-1}) = \emptyset$ for $|k|\geq 2$. 
  
  Recall the Walnut representation (Proposition \ref{pro:walnut}) of $\bd{S}$ to see that $g_n\bd{T}_{kb_n^{-1}}\overline{g_n} \equiv 0$ for $|k| \geq 2$. The support of products of shifted weights $\bd{T}_x \omega_{n,k}$ will play a substantial role in proving Theorem \ref{thm:mainres}. Indeed, they are the motivation behind the definition of the intervals $I^\pm_{n,k}$. Since a better understanding of their relations in the setting of Theorem \ref{thm:mainres} is crucial, we precede the proof with a lemma discussing these relations. The results are used, or at least considered, several times during the course of the proof of Theorem \ref{thm:mainres}.

\begin{Lem}\label{lem:intervals}
    Under the conditions of Theorem \ref{thm:mainres}, the following hold for all $n,m\in\ZZ, k,j\in\NN$:
    \begin{itemize}
    \item[(a)] $I^-_{n,k+1} = I^-_{n,1} \cap (I^-_{n+1,k}-b_n^{-1})$, with $|I^-_{n,k+1}| < \min\{|I^-_{n,k}|,|I^-_{n+1,k}|\}$.
		Analogous: $I^+_{n,k+1} = I^+_{n,1} \cap (I^+_{n-1,k}+b_n^{-1})$, with $|I^+_{n,k+1}| < \min\{|I^+_{n,k}|,|I^+_{n-1,k}|\}$.\vspace{4pt}
    \item[(b)] $I^\pm_{n,k+1}\subsetneq I^\pm_{n,k}$.\vspace{4pt}
    \item[(c)] For $n\neq m$, $I^+_{n,k}\cap I^+_{m,j} = \emptyset$. Analogous: $I^-_{n,k}\cap I^-_{m,j} = \emptyset$.\vspace{4pt}
    \item[(d)] Whenever $I^-_{n,k}\cap I^+_{m,j} \neq \emptyset$, it follows that $m\in\{n-1,n-2\}$. Furthermore $I^-_{n,k}\cap I^+_{n-2,j}\neq \emptyset$ implies $c_n=d_{n-2}$.\vspace{4pt}
    \item[(e)] For $m\neq n$, $I^-_{n,k} + b^{-1}_n\cap I^+_{m,j} = \emptyset$ and 
		$I^-_{n,k}\cap I^+_{m,j} - b^{-1}_m = \emptyset$.\vspace{4pt}
    \item[(f)] The following are equivalent:
		\begin{enumerate}
		\item[(i)] $I^-_{n-2,0} + b^{-1}_{n-2}\cap I^-_{n,0} \neq \emptyset$,
		\item[(ii)] $I^+_{n-2,0}\cap I^+_{n,0} - b^{-1}_{n} \neq \emptyset$,
		\item[(iii)] $c_n = d_{n-2}$.
		\end{enumerate}
    \end{itemize}
\end{Lem}
\begin{proof}
    (a) The conditions on $\bd{b}$ imply $c_n> c_{n+1}-b^{-1}_n$ and $d_{n+k+1} < d_{n+k}+b^{-1}_{n+k+1}$, proving the statement about the size of $I^+_{n,k+1}$. Further, by the same argument, $I^-_{n,1} \cap (I^-_{n+1,k}-b_n^{-1}) = [c_n,d_n-b_n^{-1}]\cap [c_{n+1}-b^{-1}_n,d_{n+k+1}-B^+_{n,k+1}] = [c_n,d_{n+k+1}-B^+_{n,k+1}] = I^-_{n,k+1} = I^-_{n,k} \cap (I^-_{n+1,k}-b_n^{-1})$. The proof for $I^+_{n,k+1}$ is analogue.\vspace{4pt}
	      
    (b) Follows from (a).

    (c) By (a), it is sufficient to show that $I^+_{n,1}\cap I^+_{m,1} = \emptyset$ and $I^-_{n,1}\cap I^-_{m,1} = \emptyset$ for $n\neq m$. Since the conditions on $\bd{b}$ guarantee 
	$d_n-b^{-1}_n < d_{n-1} \leq c_{n+1} < c_n+b^{-1}_n$, (c) is immediate.\vspace{4pt}

    (d) Assume $m < n-1$ or $m > n$, then it is easy to see that $I^-_{n,1}\cap I^+_{m,1} = \{c_n\}$ if $m = n-2$, $c_n = d_{n-2}$ and otherwise $I^-_{n,1}\cap I^+_{m,1} = \emptyset$. For $m = n$ we get 
	$I^-_{n,1}\cap I^+_{m,1} = \emptyset$ by the conditions on $\bd{b}$. The second part immediately follows from (b) with $I^-_{n,1} + b^{-1}_n = I^+_{n,1}$, the third part is analogue.\vspace{4pt}

    (e) and (f) follow from (b),(c), resp. (c),(d), together with $I^-_{n,1} + b^{-1}_n = I^+_{n,1}$.%\vspace{4pt}
\end{proof}

Lemma \ref{lem:intervals}(f) and the second part of (d) are concerned with the case that $c_n = d_{n-2}$ for some $n\in\ZZ$. Considering these points in the following proof would lead to some weights that are non-zero on a countable set only. Since their essential support is empty, they can be ignored when considering operators mapping $\LtR$ to $\LtR$. To avoid the treatment of these values altogether, we will, without loss of generality, assume $g_n(c_n) = 0$ for all $n\in\ZZ$. However, when considering discrete NSG systems, these ``point weights'' influence the action of the corresponding operator and must be considered, somewhat complicating the argument. For more information regarding that case, see Remark \ref{rem:discnsg}.

  With Lemma \ref{lem:intervals} in place, we can now proceed to the proof of Theorem \ref{thm:mainres}.\\
  
  \begin{proof}[Proof of Theorem \ref{thm:mainres}]
	(i): If $b_n \leq \frac{1}{d_n-c_n}$ for all $n\in\ZZ$, then the frame operator $\bd{S}$ is diagonal and there is nothing to prove, cf. \cite[Theorem 1]{badohojave11} for more information. Otherwise, we make use of both the Neumann series (\ref{eq:neumann}) representation of $\bd{S}^{-1}$ and the Walnut representation of $\bd{S}$. Since we assume $\mathcal{G}(\bd{g},\bd{b})$ to be a frame with frame bounds $A,B$, the Neumann series converges to the inverse frame operator and each of its elements defines a bounded, linear operator. The proof can roughly be structured into two parts. First, we use an induction argument to show that each element $\bd{N}^j$, $j\in\NN_0$ with $\bd{N}:=(\bd{I}-2\bd{S}/(A+B))$, of the Neumann sum $\frac{2}{A+B}\sum_{j=0}^\infty \bd{N}^j$ possesses a Walnut-like representation of the form 
	\begin{equation}\label{eq:induct}
	  \bd{N}^{j}f = \omega_{j,0}f + \sum_{n\in\ZZ}\sum_{k=1}^j \omega_{j,n,k} \bd{T}_{-B^{+}_{n,k}}f + \omega_{j,n,-k} \bd{T}_{B^{-}_{n,k}}f,
	\end{equation}
	for all $f\in\LtR$, with $\supp(\omega_{j,n,k}) \subseteq I^{-\sgn(k)}_{n,|k|}$ for all $j\in\NN_0,~k\in\NN$. Since $\bd{S}^{-1} = \sum_{j\in\NN_0} \bd{N}^{j}$. The second part discusses convergence of the sum of Walnut-like representations to the desired Walnut-like representation of $\bd{S}^{-1}$.
	
	As noted above, we assume $g_n(c_n) = 0$ for all $n\in\ZZ$. This is no restriction since $g_n \equiv g_n\chi_{]c_n,d_n]}$ in $\LtR$. Since the identity operator $\bd{I}$ has a Walnut-like representation $\bd{I}f = \omega_{0,0} f$ with $\omega_{0,0} \equiv 1$, we can invoke the conditions on $\mathcal{G}(\bd{g},\bd{b})$ to see that
	\begin{equation*}
	  \bd{N}^{0}f = f \text{ and } \bd{N}^{1}f = \omega_{1,0}f + \sum_{n\in\ZZ} \omega_{1,n,1} \bd{T}_{-B^+_{n,1}}f + \omega_{1,n,-1} \bd{T}_{B^-_{n,1}}f,
	\end{equation*}
	for all $f\in\LtR$ , with $\supp(\omega_{1,n,1}) \subseteq I^-_{n,1}$ and $\supp(\omega_{1,n,-1}) \subseteq I^+_{n,1}$. This proves \eqref{eq:induct} for $j\in\{0,1\}$.
	
	For the induction step, we show that \eqref{eq:induct} for $j$ implies \eqref{eq:induct} for $j+1$ for all $j\in\NN$. We define for $j \geq 1$
		\begin{equation*}
		  \bd{N}_D^{j}f = \omega_{j,0}f + \sum_{n\in\ZZ}\sum_{k=1}^{j-1} \omega_{j,n,k} \bd{T}_{-B^{+}_{n,k}}f + \omega_{j,n,-k} \bd{T}_{B^{-}_{n,k}}f,
		\end{equation*}
		and 
		\begin{equation*}
		  \bd{N}_R^{j}f = \sum_{n\in\ZZ} \omega_{1,n,j} \bd{T}_{-B^+_{n,j}}f + \omega_{1,n,-j} \bd{T}_{B^-_{n,j}}f.
		\end{equation*}

		This allows us to write $\bd{N}^{j}$ as the sum of $\bd{N}_D^{j}$ and $\bd{N}_R^{j}$ and consequently
		\begin{equation}\label{eq:Si2}
		  \bd{N}^{j+1} = \bd{N}^{j}\bd{N}^{1} = \bd{N}_D^{j}\bd{N}^{1} + \bd{N}_R^{j}\bd{N}^{1}.
		\end{equation}
		
		Note that the only assumptions made on the form of $\bd{N}^{j}$ is the support of the weights, allowing us to use the induction assumption to show that $\bd{N}_D^{j}\bd{N}^{1}$ has a Walnut-like representation of the form 
		\begin{equation}\label{eq:Sijdiag}
		  \bd{N}_D^{j}\bd{N}^{1}f = \eta_{j,0}f + \sum_{n\in\ZZ}\sum_{k=1}^{j} \eta_{j,n,k} \bd{T}_{-B^{+}_{n,k}}f + \eta_{j,n,-k} \bd{T}_{B^{-}_{n,k}}f
		\end{equation}
		for all $f\in\LtR$, with $\supp(\eta_{j,n,k}) \subseteq I^{-\sgn(k)}_{n,|k|}$ for all $k\in\{1,\ldots,j\}$. 
		
		On the other hand, for all $f\in\LtR$, $\bd{N}_R^{j}\bd{N}^{1}f$ can be written as
		\begin{equation*}
		  \begin{split}
		  \lefteqn{\bd{N}_R^{j}\bd{N}^{1}f =} \\ & \sum_{n\in\ZZ}\sum_{|k|=j} \omega_{j,n,k} \bd{T}_{-\sgn(k)B^{\sgn(k)}_{n,j}} \left(\omega_{1,0}f + \sum_{\tilde{n}\in\ZZ}\sum_{|l|=1} \omega_{1,\tilde{n},l} \bd{T}_{-\sgn(l)B^{\sgn(l)}_{\tilde{n},1}}f\right).
		  \end{split}
		\end{equation*}
		
		By Lemma \ref{lem:intervals} and for all $n\in\ZZ,~|k| = j$:
		\begin{align*}
		  \supp(\omega_{j,n,k}\bd{T}_{-\sgn(k)B^{\sgn(k)}_{n,j}}\omega_{1,0}) & \subseteq I^{-\sgn(k)}_{n,j},\\
		  \supp(\omega_{j,n,k}\bd{T}_{-\sgn(k)B^{\sgn(k)}_{n,j}}\omega_{1,\tilde{n},l}) & \subseteq 
		  \begin{cases} I^{-\sgn(k)}_{n,j} & ,\ l=-\sgn(k),~\tilde{n}=n+k+l,\\
				I^{-\sgn(k)}_{n,j+1} &,\ l=\sgn(k),~\tilde{n}=n+k,\\
				\emptyset & \text{, else.}
		  \end{cases}
		\end{align*}
		Order the appearing weights by the corresponding translate of $f$ and take their sum to find that $\bd{N}_R^{j}\bd{N}^{1}f$ can be written as
		\begin{equation}\label{eq:Sijside}
		  \bd{N}_R^{j}\bd{N}^{1}f = \widetilde{\omega_{j+1,0}}f + \sum_{n\in\ZZ}\sum_{k=j-1}^{j+1} \widetilde{\omega_{j+1,0,k}} \bd{T}_{-B^{+}_{n,k}}f + \widetilde{\omega_{j+1,0,-k}} \bd{T}_{B^{-}_{n,k}}f
		\end{equation}
		for all $f\in\LtR$, with $\supp(\widetilde{\omega_{j+1,n,k}}) \subseteq I^{-\sgn(k)}_{n,|k|}$ for $k\in\{\pm (j-1),\pm j,\pm (j+1)\}$. Considering Equations \eqref{eq:Sijdiag} and \eqref{eq:Sijside}, we conclude that \eqref{eq:induct} holds for $j+1$ completing the induction argument.

		We combine the results so far, arriving, for all $f\in\LtR$, at
		\begin{align}\label{eq:nowinterchange}
		  \lefteqn{\bd{S}^{-1}f = \frac{2}{A+B}\sum_{j\in\NN_0} \bd{N}^{j}f} \nonumber \\
		  & = \frac{2}{A+B}\left(\sum_{j\in\NN_0} \omega_{j,0}f + \sum_{n\in\ZZ}\sum_{k\in\ZZ\setminus\{0\}} \omega_{j,n,k}\bd{T}_{-\sgn(k)B^{\sgn(k)}_{n,|k|}}f\right).
		\end{align}
		Recall that $\|\bd{N}^{1}\|_{op} \leq C$, with $C = \frac{B-A}{B+A} < 1$, where $A$ and $B$ are the optimal frame bounds for $\mathcal{G}(\bd{g},\bd{b})$. To conclude the proof, we want to interchange the sum over $j$ with the sums over $n$ and $k$. To achieve that, we show absolute convergence in operator norm. Observe that, by $d_n \leq c_{n+2}$ we see that $\sum_{n\in\ZZ} \|f\mid_{[c_n,d_n]}\| \leq 2\|f\|$ for all $f\in\LtR$ and, checking the support of $\omega_{j,n,k}$, we have 
		\begin{equation*}
		  \omega_{j,n,k}\bd{T}_{-\sgn(k)B^{\sgn(k)}_{n,|k|}}f = \omega_{j,n,k}\bd{T}_{-\sgn(k)B^{\sgn(k)}_{n,|k|}}(f\mid_{[c_{n+k-\sgn(k)},d_{n+k-\sgn(k)}]}).
		\end{equation*}
		By Lemma \ref{lem:sepeqbnd}, we have $|\omega_{j,0}|,|\omega_{j,n,k}| \leq C^j$ and consequently
		\begin{align*}
		  \lefteqn{\left\| \sum_{j\in\NN_0} |\omega_{j,0}|f + \sum_{n\in\ZZ}\sum_{k\in\ZZ\setminus\{0\}} |\omega_{j,n,k}|\bd{T}_{-\sgn(k)B^{\sgn(k)}_{n,|k|}}f \right\|}\\
		  & \leq \left\| \sum_{j\in\NN_0} C^j|f| + \sum_{n\in\ZZ}\sum_{|k|\leq j} C^j \bd{T}_{-\sgn(k)B^{\sgn(k)}_{n,|k|}}\left|f\mid_{[c_{n+k-\sgn(k)},d_{n+k-\sgn(k)}]}\right|\right\|. \tag{\textasteriskcentered}
		\end{align*} %\\
		In the next step, we separate the first term in the sum over $j$ and for the remaining terms, interchange the sums over $j$ with that over $n$. We also reorder the sums over $j$ and $k$ by the appearing restrictions of $f$.
		\begin{align*}
		  \lefteqn{\text{(\textasteriskcentered)}}\\ & = \left\| \frac{1}{1-C}|f| + \sum_{n\in\ZZ}\sum_{k\in\ZZ\setminus\{0\}} \bd{T}_{-\sgn(k)B^{\sgn(k)}_{n,|k|}}\left|f\mid_{[c_{n+k-\sgn(k)},d_{n+k-\sgn(k)}]}\right| \sum_{j\geq |k|} C^j \right\| \\
		  & = \left\| \frac{1}{1-C}|f| + \frac{1}{1-C}\sum_{n\in\ZZ}\sum_{k\in\ZZ\setminus\{0\}} C^{|k|} \bd{T}_{-\sgn(k)B^{\sgn(k)}_{n,|k|}}\left|f\mid_{[c_{n+k-\sgn(k)},d_{n+k-\sgn(k)}]}\right| \right\| \\
		  & \leq  \frac{1}{1-C}\left\|f\right\| + \frac{2}{1-C}\sum_{n\in\ZZ}\sum_{k\in\NN} C^k \left\|f\mid_{[c_n,d_n]}\right\| \\
		  & \leq \frac{1}{1-C}\left\|f\right\| + \frac{4}{1-C}\left\|f\right\| \sum_{k\in\NN} C^k  = \frac{3C+1}{(1-C)^2}\left\|f\right\|.
		\end{align*}
		Hence, we can interchange the sums in \eqref{eq:nowinterchange} and find
		\begin{equation}
		  \bd{S}^{-1}f = \frac{2}{A+B} \left(\omega_0 f + \sum_{n\in\ZZ}\sum_{k\in\ZZ\setminus\{0\}} \omega_{n,k}\bd{T}_{-\sgn(k)B^{\sgn(k)}_{n,|k|}} f\right),
		\end{equation}
		with $\omega_0 = \sum_j \omega_{j,0}$, $\omega_{n,k} = \sum_j \omega_{j,n,k}$ and $\|\omega_0\|_\infty \leq \frac{1}{1-C}$, $\|\omega_{n,k}\|_\infty \leq \frac{C^{|k|}}{1-C}$ for all
		$n\in\ZZ$, $k\in\ZZ\setminus\{0\}$. This concludes the proof of (i).                

	(ii): To determine the support of $\widetilde{g_{m,n}}$, we use (i) and collect the weights $\omega_{\tilde{n},k}$ such that $\supp(\bd{T}_{\sgn(k)B^{\sgn(k)}_{n,|k|}}\omega_{\tilde{n},k}) \cap [c_n,d_n] \neq \emptyset$. By checking the support properties, these can be found to to be exactly the weights 
	\begin{align*}
	  \omega_{n-k+1,k},\ \omega_{n+k-1,-k} \text{, for all } n\in\ZZ,~k\in\NN \text{ and}\\
	  \omega_{n-k,k},\ \omega_{n+k,-k} \text{, for all } n\in\ZZ,~k\in\NN.
	\end{align*}
	Consequently, 
	\begin{equation}\label{eq:formcandual}
	  \begin{split}
	  \widetilde{g_{m,n}} = \omega_0 g_{m,n} + \sum_{k\in\NN} \omega_{n-k+1,k}\bd{T}_{-B^+_{n-k+1,k}}g_{m,n} + \omega_{n-k,k} \bd{T}_{-B^+_{n-k,k}}g_{m,n}\\
	  \hspace{20pt} + \omega_{n+k-1,-k} \bd{T}_{B^-_{n+k-1,k}}g_{m,n} + \omega_{n+k,-k} \bd{T}_{B^-_{n+k,k}}g_{m,n}
	  \end{split}
	\end{equation}
	holds and
	\begin{equation*}
	  \supp(\widetilde{g_{m,n}}) = I_{n,0} \bigcup_{k\in\NN} I^-_{n-k+1,k} \cup I^-_{n-k,k} \cup I^+_{n+k-1,k} \cup I^+_{n+k,k}.
	\end{equation*}
	Complete the proof of (ii) by noting that $I^{\pm}_{n,k+1} \subseteq I^{\pm}_{n,k}$ and $I^\pm_{n,1} \subseteq I_{n,0}$, where we applied Lemma \ref{lem:intervals}.
	
	(iii): We know that $g_{m,n} = \bd{M}_{mb_n}g_n$ and $e^{2\pi imbnt}$ %$\exp(2\pi imb_nt)$
	is a $b^{-1}_n$-periodic function. Furthermore, $B^-_{n+k,k+1} = B^-_{n+k,k} + b^{-1}_n$ and analogous for $B^+_{n-k,k+1}$. Apply Equation \eqref{eq:formcandual} to $g_{0,n} = g_n$ and $g_{m,n} = g_n e^{2\pi imbn\cdot}$ to confirm (iii).

	(iv): Fix $n\in\ZZ$ and $k_n^-\in\NN_0$ the smallest integer, such that $d_n-c_n-b^{-1}_n \leq k_n^- \epsilon$. Since $c_m+b_m^{-1} \geq c_{m+1}+\epsilon$, for all $m\leq n$, we see that $I^-_{n,\tilde{k}} = \emptyset$ for all 
	$\tilde{k}\geq k_n^-$. Analogous, there exists $k_n^+\in\NN_0$, such that $I^+_{n,\tilde{k}} = \emptyset$ for all $\tilde{k}\geq k_n^+$. Define $k_n:= \max\{k_n^+,k_n^-\}$, then $I^\pm_{n,\tilde{k}} = \emptyset$ for all $\tilde{k}\geq k_n$. This proves the first part. To prove the second part note that $|I^+_{n,0}| = |I^-_{n,0}| \leq C$ follows from $d_n-c_n-b_n^{-1} \leq C$ for all $n\in\ZZ$. As before, $|I^\pm_{n,k+1}|\leq |I^\pm_{n,k}|-\epsilon$, concluding the proof.
    \end{proof}
    
    \begin{Rem}
      Altogether, Theorem \ref{thm:mainres} tells us that, in the described case, the canonical dual frame of $\mathcal{G}(\bd{g},\bd{b})$ is not too different in structure from $\mathcal{G}(\bd{g},\bd{b})$ itself. 
      From Theorem \ref{thm:mainres}(iii) in particular, we see that a choice of constant $b_n$ leads, as expected, to a canonical dual that is also a nonstationary Gabor frame. In cases where $b_n$ varies in a systematic
      way, e.g. as powers of $2$, $\{\bd{S}^{-1}g_{n,m}\}_{n,m\in\ZZ}$ can be interpreted as a nonstationary Gabor frame such that the functions $\bd{S}^{-1}g_{n,m}$ for fixed $n\in\ZZ$ are constructed from few prototypes by regular modulation with some step $\tilde{b}_n$ only dependent on $n$.
    \end{Rem}
    
    \begin{Rem}\label{rem:discnsg}[Discrete NSG systems]
      Above, we have disregarded systems with $g_n(d_n)g_{n+2}(d_n) \neq 0$, because isolated points are null sets in $\LtR$. This is not anymore true in $\ell^2(\ZZ)$ or $\CC^L$ with the usual point measure. Hence, these cases are worth some consideration. By considerations similar to those in the proof above, additional weights $\rho^\pm_{n,k,l}$ may appear for $n\in\ZZ,k\in\NN,l\in\NN_0$. All of them are supported on a single point, more explicitly $\supp(\rho^+_{n,k,l}) = d_n-B^+_{n-k+1,k}$ with corresponding translation operator $\bd{T}_{-B^+_{n-k+1,k}-B^+_{n+2,l+1}}$ and $\supp(\rho^-_{n,k,l}) = d_n+B^-_{n+k+1,k}$ with corresponding translation operator $\bd{T}_{B^-_{n+k+1,k}+B^-_{n,l+1}}$. 
      
      However, also in the discrete setting, ``smooth'' window functions are preferred for their better time-frequency concentration. Therefore, assuming $g_n$ to be zero at the endpoints of its support is a weak restriction.
      
      With this caveat and the usual considerations in mind, the proof of Theorem \ref{thm:mainres} above can be directly applied to NSG systems in $\ell^2(\ZZ)$.
    \end{Rem}

    \begin{Rem}[Frames for $\CC^L$]\label{rem:CL}
      For finite, discrete nonstationary Gabor transforms, Theorem \ref{thm:mainres} applies with essentially the obvious adjustments. Albeit, the circular nature of this setting introduces potential complications. 
      
      To ensure that (iii) still holds, we must guarantee that the intervals $I^{(0)}_n$, $I^+_{n-k,k}$, $I^-_{n+l,l}$ are disjoint. Assume the number of windows $g_n$ to be $N$. Then if the nonstationary Gabor system in question satisfies $g_n(d_n)g_{n+2}(d_n) = 0$ for all $n\in\{0,\ldots,N-1\}$, it is sufficient that $\sum_{n=0}^{N-1} b_n^{-1} \geq L + \max_n |I^+_{n,1}|$.

      Fast computation of the inverse frame operator can be implemented e.g. via a structured Gaussian elimination algorithm.
    \end{Rem}  
  
  \section{Conclusion}\label{sec:conclu}
    We presented several results on the structure of nonstationary Gabor systems with low redundancy in time and moderate redundancy in frequency, demonstrating that such systems, if invertible, possess an inverse frame operator with a distinct structure not too different from that of the original frame operator. While the canonical dual frame will be of nonstationary Gabor type only if the modulation parameters $b_n$ are chosen uniformly, we have given a simple condition on the existence of a dual nonstationary Gabor frame satisfying the exact same support conditions. Furthermore, such a frame can be constructed by solving a simple set of equations. Reduction of our results to the case of classical Gabor systems shows that the canonical dual frame satisfies a special support condition, for which Christensen, Kim and Kim have recently shown the existence of a dual frame satisfying it. Under stronger restrictions on the redundancy of the Gabor system, we showed that this support condition can be improved to 
coincide with the original support.
    
    Further, we have generalized the duality conditions for Gabor systems to the setting of well-behaved NSG systems, providing a tool for investigating the existence of dual pairs of nonstationary Gabor systems. 
    
    Future work includes the investigation of the inverse frame operator for more general NSG systems, allowing for higher overlap and/or coarser frequency sampling, although numerical experiments have shown that low redundancy systems with high overlap possess a highly non-sparse inverse frame operator. Moreover, harnessing the results in this manuscript to provide fast implementations for the inversion of certain discrete nonstationary Gabor frames, extending the flexibility of such systems in applications is planned.

  \section{Thanks and Acknowledgement}
    The author would like to thank Peter Balazs, Monika D\"{o}rfler, Ewa Matusiak, Peter L. S\o{}ndergaard and Christoph Wiesmeyr for proofreading and fruitful discussion on the topics of the manuscript.

    This research was supported by the EU FET Open grant UNLocX (255931) and the Austrian Science Fund (FWF) START-project FLAME (``Frames and Linear Operators for Acoustical Modeling and Parameter Estimation''; Y 551-N13).

\end{document}